\documentclass[11pt]{amsart}
\usepackage[T1]{fontenc}
\usepackage[utf8]{inputenc}
\usepackage[UKenglish]{babel}
\usepackage[a4paper,margin=29mm,headheight=14pt]{geometry}
\usepackage{lmodern,microtype,amssymb,mathtools,amsthm}
\usepackage{enumitem,hyperref}
\hypersetup{hidelinks,pdftitle={Voevodsky's motives and zeroth slice of motivic sphere},pdfauthor={Bo Zhang},pdfsubject={Nisnevich motives as modules over the zeroth slice after inverting residue characteristics}}
\numberwithin{equation}{section}
\newtheorem{theorem}{Theorem}[section]
\newtheorem{proposition}[theorem]{Proposition}
\newtheorem{lemma}[theorem]{Lemma}
\newtheorem{corollary}[theorem]{Corollary}
\theoremstyle{definition}

\theoremstyle{remark}
\newtheorem{remark}[theorem]{Remark}
\title[Voevodsky's motives and the zeroth slice]{Voevodsky's motives and zeroth slice of motivic sphere}
\author{Bo Zhang}
\email{bozhang9511@gmail.com}
\begin{document}
\raggedbottom
\begin{abstract}
Let \(S\) be a noetherian scheme of finite Krull dimension and let
\(R\) be a commutative ring in which the positive residue characteristics
of \(S\) are invertible. We identify Voevodsky's stable Nisnevich
category of motives with \(R\)-coefficients with the category of modules
over the zeroth slice of the motivic sphere with derived
\(R\)-coefficients. The equivalence takes the free module on a smooth
\(S\)-scheme to its motive with transfers and commutes with pullback
and extension of coefficients. It applies to singular and nonreduced
schemes and to mixed-characteristic bases. The main geometric input
is the base change stability of Voevodsky's motivic Eilenberg--MacLane
spectrum \(HR_S\). We prove this using permutation quotients of tori,
families of relative cycles and arithmetic power homotopies. The identification over fields and
the reconstruction theorem of Cisinski--D\'eglise then give the module
equivalence.
\end{abstract}
\maketitle
\begingroup
\small
\tableofcontents
\endgroup
\clearpage

\section{Introduction}\label{sec:introduction}

A smooth \(S\)-scheme \(X\) determines both a motivic suspension
spectrum in \(\mathrm{SH}(S)\) and a motive with finite transfers in
\(\mathrm{DM}(S,R)\). The comparison proved here realizes the passage
from the first object to the second as formation of a free module.
Its coefficient algebra comes from the slice filtration on
\(\mathrm{SH}(S)\). In particular, mapping spectra between motives
with transfers can be computed in stable motivic homotopy theory.
We work with the Nisnevich topology throughout.

\subsection{Motives as modules over the zeroth slice}

Let \(S\) be a noetherian scheme of finite Krull dimension and \(R\)
a commutative ring. We use the presentable stable \(\infty\)-categories
\(\mathrm{SH}(S)\) of motivic spectra, \(\mathrm{DA}(S,R)\) of
Nisnevich \(\mathbb A^1\)-local Tate spectra of \(R\)-complexes
without transfers, and \(\mathrm{DM}(S,R)\) of stable Nisnevich motives
with finite transfers. Write \(\mathbf1_S\) for the motivic sphere,
\(\mathbf1_{S,R}\) for the unit of \(\mathrm{DA}(S,R)\), and
\(\mathbf1_R\) for the unit of \(\mathrm{DM}(S,R)\). The graph of a
morphism defines the adjunction
\[
 F_{S,R}:\mathrm{DA}(S,R)\rightleftarrows\mathrm{DM}(S,R):G_{S,R}.
\]
The left adjoint adds finite transfers and is symmetric monoidal.
Thus \(HR_S=G_{S,R}(\mathbf1_R)\) is a commutative algebra in
\(\mathrm{DA}(S,R)\). Its underlying motivic spectrum is
Voevodsky's motivic Eilenberg--MacLane spectrum with coefficients in \(R\).

Write \(s_0(\mathbf1_S)\) for the zeroth slice of the sphere.
Bachmann--Elmanto--Morrow identify it with the absolute commutative
motivic ring spectrum \(s_0(\mathrm{KGL}_S)\)
\cite[Corollary~9.7]{BEM}. Let \(HR^{\mathrm{top}}_S\) be the image in
\(\mathrm{SH}(S)\) of the ordinary Eilenberg--MacLane spectrum of \(R\).
By adjunction over \(\operatorname{Spec}\mathbb Z\) and pullback to
\(S\), the weight-zero map of \cite[Construction~4.35 and
Remark~4.37]{BEM} gives a map of commutative algebras
\(H\mathbb Z^{\mathrm{top}}_S\to s_0(\mathbf1_S)\).
As explained in Section~\ref{sec:models}, this specifies a commutative
algebra \(A_{S,\mathbb Z}\) in \(\mathrm{DA}(S,\mathbb Z)\). Put
\[
 A_{S,R}=A_{S,\mathbb Z}\otimes_{\mathbb Z}^{\mathbf L}R
 \quad\text{in }\mathrm{DA}(S,R).
\]

Framed correspondences give a unital comparison of commutative algebras
\(\mu_{S,R}:A_{S,R}\to HR_S\)
(Proposition~\ref{prop:counit-comparison}). For an \(A_{S,R}\)-module \(N\),
\(F_{S,R}(N)\) is a module over \(F_{S,R}(A_{S,R})\) in
\(\mathrm{DM}(S,R)\). Extending scalars along the adjoint algebra map
\(F_{S,R}(A_{S,R})\to\mathbf1_R\) defines
\begin{equation}\label{eq:introduction-comparison}
 \Theta_{S,R}(N)=
 \mathbf1_R\otimes_{F_{S,R}(A_{S,R})}F_{S,R}(N)
 \quad\text{in }\mathrm{DM}(S,R).
\end{equation}

\begin{theorem}\label{thm:main}
Let \(S\) be a noetherian scheme of finite Krull dimension and let
\(R\) be a commutative ring in which every positive residue
characteristic of \(S\) is invertible. The functor
\[
 \Theta_{S,R}:
 \operatorname{Mod}_{A_{S,R}}(\mathrm{DA}(S,R))
 \xrightarrow{\ \simeq\ }\mathrm{DM}(S,R)
\]
is a symmetric monoidal equivalence of presentable stable categories.
The map \(\mu_{S,R}:A_{S,R}\to HR_S\) is an equivalence of commutative
algebras in \(\mathrm{DA}(S,R)\). These equivalences are compatible
with pullback to noetherian schemes of finite Krull dimension and with
derived extension of coefficients.
\end{theorem}

Lemma~\ref{lem:zero-slice-module-transitivity} identifies the source
of \(\Theta_{S,R}\) with modules in \(\mathrm{SH}(S)\) over
\[
 \underline{A_{S,R}}
 \simeq s_0(\mathbf1_S)\otimes_{H\mathbb Z^{\mathrm{top}}_S}
             HR^{\mathrm{top}}_S .
\]
For a smooth \(S\)-scheme \(X\), write \(K_X=\Sigma^\infty_T X_+\)
and let \(M_R(X)\) be its motive with transfers. Under this
identification, \eqref{eq:introduction-comparison} takes the free
module \(\underline{A_{S,R}}\otimes K_X\) to \(M_R(X)\). Hence
for smooth \(X,Y\) the free-module adjunction gives
\[
 \operatorname{Map}_{\mathrm{DM}(S,R)}(M_R(X),M_R(Y))
 \simeq
 \operatorname{Map}_{\mathrm{SH}(S)}
       (K_X,\underline{A_{S,R}}\otimes K_Y).
\]
Taking \(Y=S\) and allowing Tate twists and shifts of the target
represents motivic cohomology by the corresponding twists and shifts of
\(\underline{A_{S,R}}\). For integral coefficients over a
\(\mathbb Q\)-scheme the algebra is \(s_0(\mathbf1_S)\) itself.

\begin{corollary}\label{cor:special-cases}
Theorem~\ref{thm:main} applies in each of the following cases, for
\(S\) noetherian and finite dimensional.
\begin{enumerate}[label=\textup{(\roman*)}]
\item \(S\) is arbitrary and \(R\) is a commutative \(\mathbb Q\)-algebra.
\item \(S\) is a \(\mathbb Q\)-scheme and \(R\) is any commutative ring.
\item \(S\) is a \(\mathbb Z_{(p)}\)-scheme and \(R\) is a commutative
\(\mathbb Z[1/p]\)-algebra.
\item \(S\) is an \(\mathbb F_p\)-scheme and \(R\) is a commutative
\(\mathbb Z[1/p]\)-algebra.
\end{enumerate}
Here \(p\) is any prime number. Case \textup{(iii)} includes
mixed-characteristic bases.
\end{corollary}

\begin{proof}
In \textup{(i)} every prime is invertible in \(R\). In \textup{(ii)}
all residue fields have characteristic zero. In \textup{(iii)}
and \textup{(iv)} the only possible positive residue characteristic
is \(p\).
\end{proof}

\subsection{Base change for Voevodsky's motivic cohomology}

For a morphism \(f:T\to S\), the graph functor commutes with pullback.
Its right-adjoint exchange evaluated on the tensor unit gives a
canonical map of commutative algebras
\(\tau_{f,R}:f^*HR_S\to HR_T\). These maps compose under successive
changes of base. The geometric ingredient in the proof of
Theorem~\ref{thm:main} is the base change stability of \(HR\).

\begin{theorem}[Base change stability of \(HR\)]
\label{thm:coefficient-exchange}
Let \(f:T\to S\) be a morphism between noetherian schemes of finite
Krull dimension, and let \(R\) be a commutative ring in which all
positive residue characteristics of \(S\) are invertible. Then
\[
 \tau_{f,R}: f^*HR_S\xrightarrow{\ \simeq\ } HR_T
\]
is an equivalence in \(\mathrm{DA}(T,R)\).
\end{theorem}

Theorem~\ref{thm:coefficient-exchange} is proved in
Section~\ref{sec:general-exchange}. Section~\ref{sec:proof-module-comparison}
then deduces Theorem~\ref{thm:main} from the comparison over fields and
the reconstruction theorem of Cisinski--D\'eglise.

\subsection{Relation to previous work}

Over a field of characteristic zero, Voevodsky identified the zeroth
slice of the sphere with motivic cohomology
\cite[Theorem~6.6]{VoevSlice}.
Levine's homotopy coniveau tower established the corresponding slice
identifications over perfect fields \cite[Theorems~6.4.2 and~10.5.1]{Levine}.
R\"ondigs--\O{}stv{\ae}r proved the equivalence between motives and
modules over motivic cohomology in characteristic zero
\cite[Theorem~1]{RO}. Hoyois--Kelly--\O{}stv{\ae}r extended the
module comparison to perfect fields of positive characteristic after
inverting the characteristic exponent \cite[Theorem~5.8]{HKO}.

For schemes over a perfect field, Cisinski--D\'eglise proved that, with the
characteristic exponent inverted, modules over motivic cohomology
recover Nisnevich motives on regular noetherian finite-dimensional
bases \cite[Theorem~3.1]{CDIntegral}. Over possibly singular
schemes, their general comparison identifies modules over the
spectrum pulled back from the ground field with
\(\mathrm{DM}_{\mathrm{cdh}}\)
\cite[Theorem~5.1]{CDIntegral}. With rational coefficients, they
identify Beilinson motives with Nisnevich motives with transfers over
excellent geometrically unibranch schemes
\cite[Theorem~16.1.4 and Corollary~16.1.7]{CD}. These statements
specify both the base hypotheses and the descent topology of the
comparison categories.

Spitzweck constructed a motivic cohomology spectrum over Dedekind
domains of mixed characteristic and studied its categories of modules
\cite{Spitzweck}. For a general scheme \(S\) with structure morphism
\(p:S\to\operatorname{Spec}\mathbb Z\), Spitzweck's spectrum is
defined by \(M\mathbb Z_S=p^*M\mathbb Z_{\mathbb Z}\)
\cite[Remark~4.14]{BEM}. Bachmann--Hoyois identified it with
the zeroth slice of the sphere over schemes essentially smooth over
a Dedekind domain, in particular over \(\operatorname{Spec}\mathbb Z\)
\cite[Theorem~B.4]{BHNorms}. Hoyois also identified \(M\mathbb Z_S\)
with the framed suspension of the constant sheaf of integers with
degree transfers \cite[Theorem~21]{HoyoisLocalization}.

For every quasicompact quasiseparated scheme \(S\),
Bachmann--Elmanto--Morrow proved that the zeroth slice of the sphere
is stable under pullback and that the unit \(\mathbf1_S\to\mathrm{KGL}_S\)
induces an equivalence on zeroth slices
\cite[Corollary~9.7]{BEM}. Pulling back the equivalence over
\(\operatorname{Spec}\mathbb Z\) therefore gives an integral equivalence
\[
 M\mathbb Z_S=p^*M\mathbb Z_{\mathbb Z}
 \simeq p^*s_0(\mathbf1_{\mathbb Z})
 \simeq s_0(\mathbf1_S).
\]
Modules over this zeroth slice carry a six-functor formalism
\cite[Corollary~9.10]{BEM}.

Theorem~\ref{thm:main} identifies these module categories, after derived
extension to \(R\), with Voevodsky's stable Nisnevich motives
\(\mathrm{DM}(S,R)\). In the stated coefficient range the comparison
applies to singular and nonreduced schemes, including bases of mixed
characteristic, without excellence or finite-type assumptions.
The base change stability of \(HR\) in
Theorem~\ref{thm:coefficient-exchange} carries the comparison with
Voevodsky's motivic Eilenberg--MacLane spectrum from the residue fields
of the arithmetic base to all these schemes.

\subsection{Proof of base change stability}

Section~\ref{sec:models} fixes the arithmetic base \(B\) and
coefficients \(\Lambda\), constructs \(\mu_{S,R}\), and gives the
finite-cycle model for \(H\Lambda_S\). At rank \(r\) let
\(\mathcal Z_{S,r}\) be the sheaf of universal finite relative cycles
on \(\mathbb G_{m,S}^{r}\). Resetting any one coordinate to \(1\)
defines a split idempotent. Taking the summand reduced in all \(r\)
coordinates and applying a shift and a Tate twist gives a normalized
stage \(E_{S,r,\Lambda}\). These objects form a sequential
presentation
\[
 H\Lambda_S\simeq\operatorname*{colim}_{r\geq 0} E_{S,r,\Lambda}.
\]
The cycle pullbacks give canonical stage maps whose colimit is the
exchange \(\tau_{f,\Lambda}\). Thus base change for the cycle
stages on seminormal bases gives the first case of
Theorem~\ref{thm:coefficient-exchange}.

To compare cycles on seminormal bases, one can use the group completion
of the monoid of effective families, presented by the symmetric powers
of \(\mathbb G_m^r\) after localization at cycle multiplicities.
The geometric problem is already present at one symmetric degree
\(n\). Write \(Y=(\mathbb G_m^r)^n\) with \(\Sigma_n\) acting by
permutation of the blocks. The geometric quotient commutes with base
change in this model, but the linear motive of its restriction to
the smooth site has a comparison map whose cofiber is
\(\mathcal B_f(Y,\Sigma_n)\). Section~\ref{acg-section} studies this
cofiber by induction on stabilizers. Excision localizes the calculation
near a fixed locus. On the locus where the group order is invertible,
averaging gives vanishing. Over a residue characteristic \(p\), formulas
in elementary symmetric polynomials produce lifts of powers of
Frobenius. Homotopies between such lifts show that the telescope of
the restriction of \(\mathcal B_f(Y,\Sigma_n)\) to the \(p\)-fiber,
under coordinatewise \(p\)-power, vanishes.

Section~\ref{sec:cycle-presentation} uses Rydh's theory of families of
cycles to bound the denominators in each symmetric degree. These
multiplicities are distinct from the coordinate power maps in the
stabilizer argument. Group completion and additive delooping carry
the geometric calculation to the additive cycle comparison. The
coordinate \(p\)-power acts by \(p^r\) on the fully reduced summand
of a rank-\(r\) torus, by an integral polynomial homotopy. As \(p\)
is invertible in \(\Lambda\), this action is an equivalence and the
vanishing of the telescope implies the vanishing of the restricted
additive cofiber. The tame calculation handles the generic fiber.
Residue-point conservativity gives the canonical stage comparisons
on seminormal bases, and their compatibility with the Tate bonds
gives base change stability of \(H\Lambda\) over those bases.

Section~\ref{sec:general-exchange} starts with these seminormal
comparisons. A finite-type \(B\)-scheme has finite seminormalization
of its reduction. Invariance under finite universal homeomorphisms
then gives exchange for the structural map of every finite-type
\(B\)-scheme. Continuity passes from these comparisons to arbitrary
noetherian finite-dimensional \(B\)-schemes. Composition of exchange
maps gives the result for each \(f:T\to S\), and derived extension
of coefficients gives Theorem~\ref{thm:coefficient-exchange} for
\(R\).

\subsection{Conventions}

All schemes serving as bases are noetherian and of finite Krull
dimension unless stated otherwise. The category \(\mathrm{Sm}_S\)
consists of separated smooth \(S\)-schemes of finite type.
Representable presheaves of possibly singular schemes are restricted
to this smooth site. We use the presentable stable
\(\infty\)-category enhancements of the motivic categories.
All tensor products, inverse images, cofibers and colimits are
derived. We write \(T_R=\mathbf1_{S,R}(1)[2]\) for the Tate
suspension object.

\section{Coefficients and reduced torus stages}\label{sec:models}

\subsection{Linear coefficients and the zero-slice algebra}

Throughout this paper, $R$ is an ordinary commutative ring with unit.
Tensor products with $R$, and all subsequent changes of coefficients, are
derived.  The categories $\mathrm{DA}(S,R)$ and $\mathrm{DM}(S,R)$ use the
Nisnevich topology.  Write
\begin{equation}\label{eq:model-adjunction}
 F_{S,R}:\mathrm{DA}(S,R)
 \rightleftarrows\mathrm{DM}(S,R):G_{S,R}
\end{equation}
for the adjunction induced by graphs of morphisms.  The left adjoint is
symmetric monoidal.  Write $\mathbf1_{S,R}$ for the unit of $\mathrm{DA}(S,R)$ and
$\mathbf1_R$ for the unit of $\mathrm{DM}(S,R)$. Let
$v_{S,R}:\mathrm{DA}(S,R)\to\mathrm{SH}(S)$ be the
forgetful functor of Lemma~\ref{lem:linear-model}.  Define
$U_{S,R}=v_{S,R}G_{S,R}$. Voevodsky's motivic Eilenberg--MacLane
spectrum with coefficients in $R$ is
\[
 HR_S:=U_{S,R}(\mathbf1_R)\quad\text{in }\mathrm{SH}(S).
\]
Its canonical linear lift is $G_{S,R}(\mathbf1_R)$.
We use the same notation $HR_S$ for that specified lift when the
ambient category is $\mathrm{DA}(S,R)$.  The right adjoint
$G_{S,R}$ is lax symmetric monoidal, so in this category
\begin{equation}\label{eq:model-transfer-ring}
 HR_S=G_{S,R}(\mathbf1_R)
 \quad\text{belongs to}\quad
 \operatorname{CAlg}(\mathrm{DA}(S,R)).
\end{equation}
For $f:T\to S$, the exchange transformation for this adjunction gives a
canonical map of commutative algebras
\begin{equation}\label{eq:model-exchange}
 \tau_{f,R}:f^*HR_S
 \longrightarrow HR_T.
\end{equation}
These maps satisfy the composition law, with its natural higher coherences.
They are defined before any base-change theorem is proved.

We distinguish ordinary linear coefficients from the motivic cohomology
algebra in \eqref{eq:model-transfer-ring}.  Let
$HR^{\mathrm{top}}_S$ be the image of the ordinary Eilenberg--MacLane
spectrum $HR$ under the canonical symmetric monoidal functor
$\mathrm{Sp}\to\mathrm{SH}(S)$.  The spectrum $HR_S$
defined above is obtained from the transferred unit.  The independent
zero-slice algebra $A_{S,R}$ will be compared with its canonical
linear lift by the map constructed below.

\begin{lemma}[Ordinary linear coefficients]\label{lem:linear-model}
For every commutative ring $R$ there is a natural symmetric monoidal equivalence
\begin{equation}\label{eq:model-linear-model}
 \mathrm{DA}(S,R)
 \simeq
 \operatorname{Mod}_{HR^{\mathrm{top}}_S}(\mathrm{SH}(S)).
\end{equation}
It identifies $v_{S,R}$ with the forgetful functor on modules and is
compatible with pullback in $S$ and derived extension of coefficients.
\end{lemma}

\begin{proof}
Before motivic localization, complexes of presheaves of $R$-modules are
equivalent to presheaves of $HR$-module spectra.  The localizations imposing
Nisnevich descent and $\mathbb A^1$-invariance are monoidal.  Passing to
modules over the image of the constant algebra $HR$ commutes with these
localizations, since the localized module category is characterized by imposing
the same descent and homotopy relations on free modules.  Inverting the
Tate suspension has the same compatibility with modules, by the universal
property of tensor inversion.  This gives
\eqref{eq:model-linear-model} together with the description of
$v_{S,R}$. The constructions commute with pullback. For a ring map
$R\to R'$, extension of coefficients on presheaves corresponds to
extension along $HR\to HR'$ on module spectra. Motivic localization
and Tate stabilization preserve this identification.
\end{proof}

To specify the integral linear structure on the zeroth slice, first put
$B_0=\operatorname{Spec}\mathbb Z$. Construction~4.35 and Remark~4.37
of \cite{BEM} give a map of $E_\infty$-rings
\[
 H\mathbb Z\longrightarrow \mathbb Z(0)^{\mathrm A}(B_0).
\]
Here $H\mathbb Z$ is the ordinary Eilenberg--MacLane spectrum. By
\cite[Definition~4.15 and Remark~4.22]{BEM}, the target in weight zero
is the mapping spectrum
$\operatorname{map}_{\mathrm{SH}(B_0)}
(\mathbf1_{B_0},s_0(\mathrm{KGL}_{B_0}))$ with the multiplication
induced by that of $s_0(\mathrm{KGL}_{B_0})$. The symmetric monoidal
constant-spectrum functor $\mathrm{Sp}\to\mathrm{SH}(B_0)$ has right
adjoint $\operatorname{map}_{\mathrm{SH}(B_0)}(\mathbf1_{B_0},-)$.
Taking the adjoint on commutative algebras gives a map
\[
 H\mathbb Z^{\mathrm{top}}_{B_0}\longrightarrow s_0(\mathrm{KGL}_{B_0})
 \qquad\text{in }\operatorname{CAlg}(\mathrm{SH}(B_0)).
\]
For the structural morphism $p:S\to B_0$, pull this map back to $S$.
Corollary~9.7 of \cite{BEM} identifies
$p^*s_0(\mathrm{KGL}_{B_0})\simeq s_0(\mathrm{KGL}_S)$ and identifies
$s_0(\mathbf1_S)\simeq s_0(\mathrm{KGL}_S)$ via the unit of $\mathrm{KGL}$,
both as commutative algebras. Transport along these equivalences yields
$H\mathbb Z^{\mathrm{top}}_S\to s_0(\mathbf1_S)$.
By Lemma~\ref{lem:linear-model} this specifies a commutative algebra
\begin{equation}\label{eq:model-A-definition}
 A_S\in\operatorname{CAlg}(\mathrm{DA}(S,\mathbb Z)).
\end{equation}
Since the structure map comes from $B_0$, the construction gives
\begin{equation}\label{eq:model-A-properties}
 \underline A_S\simeq s_0(\mathbf1_S),
 \qquad f^*A_S\simeq A_T
 \quad(f:T\to S).
\end{equation}

Put
\[
 A_{S,R}=A_S\otimes_{\mathbb Z}^{\mathbf L}R.
\]
Thus $A_{S,R}$ is an absolute family of commutative algebras.
Write $\underline{A_{S,R}}=v_{S,R}(A_{S,R})$ for the underlying
commutative algebra in $\mathrm{SH}(S)$.

\begin{lemma}[Modules over the zero-slice algebra]
\label{lem:zero-slice-module-transitivity}
For every commutative ring $R$ there is a canonical symmetric monoidal
equivalence
\begin{equation}\label{eq:model-module-transitivity}
 \operatorname{Mod}_{A_{S,R}}(\mathrm{DA}(S,R))
 \simeq
 \operatorname{Mod}_{\underline{A_{S,R}}}(\mathrm{SH}(S)),
\end{equation}
where
\begin{equation}\label{eq:model-relative-coefficients}
 \underline{A_{S,R}}\simeq
 s_0(\mathbf1_S)\otimes_{H\mathbb Z^{\mathrm{top}}_S}HR^{\mathrm{top}}_S.
\end{equation}
The equivalence commutes with pullback in $S$ and with derived
extension of coefficients. In particular, for $R=\mathbb Z$ it gives
\begin{equation}\label{eq:model-zero-slice-modules}
 \operatorname{Mod}_{A_S}(\mathrm{DA}(S,\mathbb Z))
 \simeq
 \operatorname{Mod}_{s_0(\mathbf1_S)}(\mathrm{SH}(S)).
\end{equation}
\end{lemma}

\begin{proof}
By Lemma~\ref{lem:linear-model}, extension from $\mathbb Z$ to $R$
on a commutative algebra in $\mathrm{DA}(S,\mathbb Z)$ is computed by
the pushout along $H\mathbb Z^{\mathrm{top}}_S\to HR^{\mathrm{top}}_S$.
Applying this to $A_S$ and using \eqref{eq:model-A-properties}
gives \eqref{eq:model-relative-coefficients}.

Forgetting the $HR^{\mathrm{top}}_S$-linear structure on an
$A_{S,R}$-module gives a module over $\underline{A_{S,R}}$ in
$\mathrm{SH}(S)$. Conversely, any such module
has the required $HR^{\mathrm{top}}_S$-action by restriction along
the structure map $HR^{\mathrm{top}}_S\to\underline{A_{S,R}}$.
These constructions are mutually inverse by transitivity of module
categories. They identify the relative tensor products of modules,
giving the symmetric monoidal equivalence
\eqref{eq:model-module-transitivity}. Pullback carries this
structure map to the corresponding map over the new base. For a
ring map $R\to R'$, the pushout formula identifies extension on the
right with tensoring over $\underline{A_{S,R}}$ with
$\underline{A_{S,R'}}$, as on the left. Thus both compatibilities
hold. Equivalently, the category in \eqref{eq:model-module-transitivity}
is obtained from the integral module category by tensoring with $D(R)$
over $D(\mathbb Z)$.
\end{proof}

\subsection{The reduced torus tower}

Let $\mathbb Z^{\mathrm{tr}}_S(X)$ denote the Nisnevich sheaf of finite
relative cycles represented by a smooth $S$-scheme $X$, and forget its
transfers when viewing it as an ordinary sheaf.  For
$X=\mathbb G_{m,S}^r$, let $e_i$ be the endomorphism induced by setting
the $i$th coordinate equal to $1$.  The $e_i$ are commuting idempotents.
Define
\begin{equation}\label{eq:model-reduced-cycles}
 \begin{gathered}
 C_{S,r,\mathbb Z}
 =\operatorname{im}\left(
     \prod_{i=1}^r(1-e_i):
     \mathbb Z^{\mathrm{tr}}_S(\mathbb G_m^r)
     \longrightarrow
     \mathbb Z^{\mathrm{tr}}_S(\mathbb G_m^r)
   \right),\\[2pt]
 C_{S,r,R}=C_{S,r,\mathbb Z}\otimes_{\mathbb Z}^{\mathbf L}R.
 \end{gathered}
\end{equation}
The image in the first line is taken in the abelian category of Nisnevich
sheaves with transfers. To see its relation to the reduced Tate power,
let $s_i:\mathbb G_m^{r-1}\to\mathbb G_m^r$ insert $1$ in position $i$
and let $p_i$ omit that coordinate. Then $p_i s_i$ is the identity and
$e_i=(s_i)_*(p_i)_*$. Write $P_r=\prod_i(1-e_i)$ and
$L=\mathbb Z^{\mathrm{tr}}_S(\mathbb G_m^r)$. Since the projectors
commute, $P_r$ annihilates each face image $(s_i)_*\mathbb Z^{\mathrm{tr}}_S
(\mathbb G_m^{r-1})$. It induces the identity on the quotient by their
sum, since there every $e_i$ induces zero. More explicitly,
$x-P_rx$ belongs to the sum of the face images for every local section
$x$ of $L$. Quotient and projector therefore give mutually inverse maps
\[
 \operatorname{im}P_r\ \simeq\
 L\big/\sum_i\operatorname{im}(s_i)_*
 \ =\ \widetilde{\mathbb Z}^{\mathrm{tr}}_S
            ((\mathbb G_m,1)^{\wedge r}).
\]
The last quotient is the definition of the reduced represented sheaf
with transfers on this smash power. At $r=0$ the empty product is the
identity. Placing the split image in degree zero gives the corresponding
retract in the derived category. Derived extension of scalars and
motivic localization preserve this splitting.

As in \cite[Paragraph~10.0.1]{CD}, our
correspondence category with arbitrary coefficients is obtained by
$R$-linear extension of the integral correspondence category.  For
coefficient localizations $\Lambda\subset\mathbb Q$, this agrees
with the category defined by $\Lambda$-universal relative cycles, since
\cite[Remark~9.1.3(3)]{CD} identifies
$c_S(X,Y)_{\mathbb Z}\otimes\Lambda$ with
$c_S(X,Y)_\Lambda$.

Write $L_{\mathrm{Nis},\mathbb A^1}$ for effective motivic localization
of complexes without transfers and $\Sigma_T^\infty$ for Tate
stabilization.  Our convention is
$T_R=\mathbf1_{S,R}(1)[2]$.  Define the normalized stages
\begin{equation}\label{eq:model-stages}
 E_{S,r,R}
 =\Sigma_T^\infty
   L_{\mathrm{Nis},\mathbb A^1}(C_{S,r,R}[-r])(-r).
\end{equation}
Exterior product of cycles and the graph of the coordinate on
$\mathbb G_m$ provide Tate bonding maps
\begin{equation}\label{eq:model-bonds}
 b_{S,r,R}:E_{S,r,R}\longrightarrow E_{S,r+1,R}.
\end{equation}
The shifts in \eqref{eq:model-stages} are fixed by the convention on
$T_R$. The unnormalized $r$th Tate level is $C_{S,r,R}[r]$, and
tensoring its suspension spectrum with $T_R^{-r}$ gives
$\Sigma_T^\infty L_{\mathrm{Nis},\mathbb A^1}(C_{S,r,R})(-r)[-r]$.

The graph map from reduced ordinary linear representables to reduced
cycles gives
\begin{equation}\label{eq:model-normalized-graph}
 g_{S,r,R}:\mathbf1_{S,R}\longrightarrow E_{S,r,R}.
\end{equation}
Here the source is the normalization of
$\widetilde R[(\mathbb G_m,1)^{\wedge r}]$, whose stabilized image is
$\mathbf1_{S,R}(r)[r]$.  These maps satisfy
$b_{S,r,R}g_{S,r,R}=g_{S,r+1,R}$ coherently.  For every $f:T\to S$ there
are also natural exchange maps
\begin{equation}\label{eq:model-stage-exchange}
 \beta_{f,r,R}:f^*E_{S,r,R}\longrightarrow E_{T,r,R}.
\end{equation}
They commute with bonding maps, graph maps, exterior products, and
composition of pullbacks. All inverse images in this formula are derived.

\begin{proposition}[Cycle presentation of $HR_S$]
\label{prop:cycle-tower-presentation}
There is a natural equivalence
\begin{equation}\label{eq:model-cycle-colimit}
 \operatorname*{colim}_{r\geq0}E_{S,r,R}
 \simeq HR_S.
\end{equation}
Its commutative multiplication is induced by exterior products of cycles,
and its unit is induced by the maps $g_{S,r,R}$.  Under this presentation,
the colimit of \eqref{eq:model-stage-exchange} is
$\tau_{f,R}$.  Moreover, for every ring homomorphism $R\to R'$,
\begin{equation}\label{eq:model-cycle-scalars}
 E_{S,r,R}\otimes_R^{\mathbf L}R'\simeq E_{S,r,R'},
 \qquad
 HR_S\otimes_R^{\mathbf L}R'
 \simeq HR'_S.
\end{equation}
These equivalences respect the stated maps and algebra structures.
\end{proposition}

\begin{proof}
The unit Tate spectrum with transfers has $r$th level represented by the
reduced transferred Tate power $C_{S,r,R}[r]$.  Restricting this spectrum
along the graph functor gives the cycle Tate spectrum without transfers.
The comparison with the derived right adjoint is a substantive feature of
this model. The forgetful functor on Tate spectra preserves stable
$\mathbb A^1$-equivalences
\cite[Proposition~6.3.9, Paragraph~11.2.16, and Definition~11.2.17]{CD}.  Consequently
the restricted unit spectrum represents $G_{S,R}(\mathbf1_R)$.

For the underlying-object calculation, use sequential Tate spectra.
The effective forgetful right adjoint preserves filtered colimits because
its left adjoint sends the compact smooth generators to compact smooth
generators.  By adjunction it also commutes with the Tate loop functor,
since its left adjoint sends the Tate object to the transferred Tate
object.  It therefore commutes with sequential spectrification.  This
identifies the sequential cycle spectrum with the derived restriction of
the transferred unit spectrum.

For clarity, put $K_r=L_{\mathrm{Nis},\mathbb A^1}(C_{S,r,R}[r])$.
The structure map $T_R\otimes K_r\to K_{r+1}$ induces the bond from
$T_R^{-r}\otimes\Sigma_T^\infty K_r$ to
$T_R^{-r-1}\otimes\Sigma_T^\infty K_{r+1}$. For any Tate
$\Omega$-spectrum $V$ with levels $V_r$, mapping the first object into
$V$ gives the mapping space from $K_r$ to $V_r$, by the free
spectrum adjunction at level $r$. Taking the limit over $r$ imposes
precisely the compatibility with the structure maps. The result is
the space of maps from the sequential cycle spectrum to $V$.
Thus the colimit of the normalized suspended levels represents the
stabilized cycle spectrum. This proves \eqref{eq:model-cycle-colimit}.
The commutative multiplication is retained from the symmetric model.
Exterior products, graph units, and canonical exchange maps are maps of
these models, so the identifications retain their multiplicative and
functorial structures.

Effective localization, stabilization, and normalized suspension are
colimit-preserving linear functors.  The first equivalence in
\eqref{eq:model-cycle-scalars} therefore follows from the definition
of $C_{S,r,R}$.  Taking colimits gives the second equivalence.  In
particular, this coefficient compatibility follows from the explicit
cycle model. For the canonical map of ring spectra this is also
\cite[Proposition~11.2.19]{CD}.
\end{proof}

\begin{remark}\label{rem:model-complex-not-space}
The object in \eqref{eq:model-stages} is the localization of a
\emph{complex of linear sheaves}. The geometry of cycles first gives
information about the free linear motives of their underlying spaces.
Section~\ref{sec:cycle-presentation} explains how to recover the
localized complex while retaining the addition law.
\end{remark}

\subsection{The comparison from the framed degree sheaf}
\label{sec:framed-comparison}

Write \(d_S\) for the sheaf of locally constant integer-valued functions
on smooth \(S\)-schemes, with framed transfers given by degree. For a
framed span \(X\xleftarrow{p} Z\xrightarrow{q}Y\), the transfer pulls a
function back along \(q\) and sums it along the finite locally free map
\(p\), with the lengths of the geometric fibers as multiplicities.
Write \(z_S\) for the Nisnevich sheaf with transfers represented by
\(S\). We view both sheaves as complexes concentrated in degree zero.

\begin{proposition}\label{prop:counit-comparison}
For every noetherian finite-dimensional scheme \(S\) and commutative
ring \(R\), there is a canonical map of commutative algebras
\begin{equation}\label{actual:mu}
 \mu_{S,R}:A_{S,R}\longrightarrow HR_S
 \qquad\text{in }\mathrm{DA}(S,R).
\end{equation}
It is compatible with derived extension of coefficients. For a morphism
\(f:T\to S\) of such schemes, let \(a_f:f^*A_{S,R}\simeq A_{T,R}\)
denote the equivalence in \eqref{eq:model-A-properties} with
coefficients in \(R\). There is a natural homotopy
\begin{equation}\label{eq:comparison-naturality}
 \mu_{T,R}\circ a_f\simeq\tau_{f,R}\circ f^*\mu_{S,R}
\end{equation}
compatible with composition in \(f\).
\end{proposition}

\begin{proof}
We first construct the adjoint of \(\mu_{S,\mathbb Z}\). The functor
\[
 \gamma_S:\operatorname{Corr}^{\mathrm{fr}}(\mathrm{Sm}_S)
 \longrightarrow\operatorname{Cor}_S
\]
sends a finite syntomic framed span to the pushforward of its
fundamental relative cycle in \(X\times_S Y\). It preserves external
products, composes with the inclusion of ordinary morphisms to give the
graph functor, and is natural under change of base
\cite[Section~5.3]{EHKInfinite}. Restricting the sheaf \(z_S\) along
\(\gamma_S\) gives a canonical isomorphism
\[
 d_S\simeq\gamma_S^*z_S.
\]
Indeed, the sheaf represented by the base with finite transfers is the
constant Nisnevich sheaf of integers
\cite[Proposition~10.2.5]{CD}. The action of the cycle of a framed
span on sections is the degree formula above. Under this identification
the exterior product is multiplication of the two pulled-back functions
on \(X\times_S Y\). Thus the identification respects the unit and the
commutative algebra structures for Day convolution.

We use derived extension and restriction along the following linear
version of \(\gamma_S\). On the finite-correspondence side a morphism
group \(\operatorname{Cor}_S(X,Y)\) is a complex in degree zero.
On the framed side the space of correspondences is a commutative
monoid under disjoint union. The cycle map is additive, so it gives
a map from its group-completion spectrum to the Eilenberg--MacLane
spectrum of \(\operatorname{Cor}_S(X,Y)\). Extending from the free
ordinary \(H\mathbb Z\)-module on the former gives the map of linear
morphism objects. The maps respect composition and external products.
They give the restriction functor on modules over these enriched
categories and its derived extension \(\gamma_{S,!}\). By enriched
Yoneda the extension takes a free linear framed representable to the
corresponding representable with finite transfers. It is symmetric
monoidal for Day convolution. The counit gives a map of commutative algebras
\begin{equation}\label{eq:framed-degree-counit}
 \gamma_{S,!}d_S\simeq\gamma_{S,!}\gamma_S^*z_S
 \xrightarrow{\ \varepsilon\ } z_S.
\end{equation}
Here the sheaves in degree zero can be regarded as presheaves when
forming the counit. We next apply Nisnevich and \(\mathbb A^1\)
localization and Tate stabilization. Extension sends the descent and
homotopy relations on representables to the corresponding relations with
transfers and sends the pointed torus to the transferred pointed torus.
Consequently its stable value on a suspended framed object is computed
by first applying extension and then taking motivic suspension. Write
\(\Sigma^\infty_{\mathrm{fr}}\) and \(\Sigma^\infty_{\mathrm{tr}}\) for
these composites of localization and suspension in the two linear
models. The target \(\Sigma^\infty_{\mathrm{tr}}z_S\) is the transferred
unit \(\mathbf1_{\mathbb Z}\), since \(z_S\) is the sheafification of
the representable of \(S\).

We identify the stable source. Framed reconstruction is a symmetric
monoidal equivalence over \(S\) \cite[Theorem~18]{HoyoisLocalization},
so passing to modules over the ordinary integral Eilenberg--MacLane
algebra identifies the linear stable framed model with
\(\mathrm{DA}(S,\mathbb Z)\) by Lemma~\ref{lem:linear-model}.
Under this identification stabilized extension along \(\gamma_S\) is
\(F_{S,\mathbb Z}\). Indeed the two graph functors compose on smooth
representables and their morphisms, and their colimit-preserving
extensions give the same stable functor
\cite[Section~5.3.5 and Proposition~5.3.6]{EHKInfinite}.

We identify the underlying framed suspension before applying the
comparison of Hoyois. Framed correspondences form a semiadditive
category \cite[Lemma~3.2.5]{EHKInfinite}. In its category of
presheaves taking finite sums to products, simplicial suspension is
the bar construction. Indeed, suspension is the realization with
degree-\(j\) term the \(j\)-fold sum of the object, and this sum is
also its \(j\)-fold product. Realizations are computed objectwise.
Consequently the successive simplicial suspensions of the abelian
group object \(d_S\) are its ordinary deloopings \(B^j d_S\).
These form an \(\Omega\)-spectrum already on presheaves, namely the
underlying Eilenberg--MacLane spectrum of \(d_S[0]\). This
identification respects exterior multiplication by naturality for
the bilinear degree pairing. For the Tate suspension, tensoring a
coefficient module with a free module on the suspension coordinate
amounts on underlying objects to tensoring with that coordinate.
Motivic localization imposes the descent and homotopy relations on
free modules as in Lemma~\ref{lem:linear-model}. Thus the underlying
object of \(\Sigma^\infty_{\mathrm{fr}}d_S\) is the framed suspension
of the degree sheaf, with its multiplication.

Hoyois identifies this suspension, with its unit and multiplication,
with Spitzweck's motivic cohomology algebra \(M\mathbb Z_S\)
\cite[Theorem~21 and Corollary~22]{HoyoisLocalization}.
Over \(B_0=\operatorname{Spec}\mathbb Z\), Bachmann--Hoyois show
that \(M\mathbb Z_{B_0}\) is a zero slice and that the cofiber of its
unit is \(1\)-effective \cite[proof of Theorem~B.4]{BHNorms}.
Consequently the equivalence
\(s_0(\mathbf1_{B_0})\simeq M\mathbb Z_{B_0}\) is induced by the unit.
To see its multiplicative structure, localize the effective category
at the tensor ideal of \(1\)-effective objects. This localization is
symmetric monoidal, and its unit is represented by
\(s_0(\mathbf1_{B_0})\). The localized unit map is therefore a map of
commutative algebras and is an equivalence since its cofiber vanishes.
The fully faithful lax symmetric monoidal right adjoint gives the
desired equivalence of algebras in \(\mathrm{SH}(B_0)\), since
\(M\mathbb Z_{B_0}\) is already local. By definition Spitzweck's
spectrum over \(S\) is pulled back from \(B_0\)
\cite[Remark~4.14]{BEM}. Pulling back this equivalence and using
\cite[Corollary~9.7]{BEM} gives \(s_0(\mathbf1_S)\simeq M\mathbb Z_S\)
as commutative algebras.

It remains to identify the ordinary integral linear structures.
By \cite[Construction~4.35 and Corollary~6.26]{BEM}, the connective
cover of the weight-zero spectrum
\(\operatorname{map}_{\mathrm{SH}(B_0)}(\mathbf1_{B_0},\underline A_{B_0})\)
is \(H\mathbb Z\). Indeed it is the connective cover of the Nisnevich
cohomology of the constant sheaf \(\mathbb Z\) on \(B_0\), and
\(B_0\) is connected. The constant-spectrum adjunction and the
connective-cover adjunction now give
\[
 \operatorname{Map}_{\operatorname{CAlg}(\mathrm{SH}(B_0))}
 (H\mathbb Z^{\mathrm{top}}_{B_0},\underline A_{B_0})
 \simeq
 \operatorname{Map}_{\operatorname{CAlg}(\mathrm{Sp})}
 (H\mathbb Z,H\mathbb Z)\simeq *.
\]
Thus the two maps from ordinary integral coefficients agree through a
contractible space of choices. Their pullbacks give the coefficient
structures over \(S\), by \eqref{eq:model-A-properties} and
\cite[Lemmas~16, 20 and Corollary~22]{HoyoisLocalization}.
We obtain \(\Sigma^\infty_{\mathrm{fr}}d_S\simeq A_S\) as ordinary
linear commutative algebras. Suspending \eqref{eq:framed-degree-counit}
now gives \(F_{S,\mathbb Z}(A_S)\to\mathbf1_{\mathbb Z}\).
Taking its adjoint in \eqref{eq:model-adjunction} defines
\(\mu_{S,\mathbb Z}\).

For base change, the maps on the two presheaves in
\eqref{eq:framed-degree-counit} come from pullback of locally constant
functions and relative cycles. Their compatibility with restriction is
the base-change formula for degrees of finite locally free maps.
Naturality of the counit therefore gives a commutative square after
motivic suspension. On the source the map is \(a_f\), by
\cite[Lemmas~16 and~20]{HoyoisLocalization}, and on the transferred
unit it is the unit isomorphism. Taking mates under
\eqref{eq:model-adjunction} gives exactly
\eqref{eq:comparison-naturality}, with the exchange of
\eqref{eq:model-exchange} on the right. The identifications respect
composition since pullback of relative cycles and formation of mates do.
Finally, extend the adjoint map to \(R\) and take its adjoint under
\(F_{S,R}\dashv G_{S,R}\). The natural coefficient exchange on the
transferred unit is the equivalence of
\eqref{eq:model-cycle-scalars} in
Proposition~\ref{prop:cycle-tower-presentation}. This identifies the new
adjoint with derived coefficient extension of \(\mu_{S,\mathbb Z}\)
and gives the asserted compatibility for all \(R\).
\end{proof}

Extension of scalars along the adjoint
\(F_{S,R}(A_{S,R})\to\mathbf1_R\) gives a colimit-preserving
symmetric monoidal functor
\begin{equation}\label{actual:module-comparison}
 \Theta_{S,R}(M)=
 \mathbf1_R\otimes_{F_{S,R}(A_{S,R})}F_{S,R}(M),
 \qquad M\in\operatorname{Mod}_{A_{S,R}}(\mathrm{DA}(S,R)).
\end{equation}
It has values in \(\mathrm{DM}(S,R)\). The adjoint form of
\eqref{eq:comparison-naturality} makes \(\Theta\) compatible with
pullback. For a smooth morphism, the projection formula and the
compatibility of \(F\) with smooth pushforward give the corresponding
smooth pushforward compatibility. On a free module one obtains
\begin{equation}\label{eq:comparison-free-modules}
 \Theta_{S,R}(A_{S,R}\otimes M)\simeq F_{S,R}(M).
\end{equation}

\subsection{The arithmetic base and the two linear constructions}
\label{sec:arithmetic-setup}

Fix a set \(P\) of prime numbers and set
\begin{equation}\label{aca-arithmetic-base}
 B=\operatorname{Spec}\mathbb Z[\ell^{-1}:\ell\notin P],
 \qquad \Lambda=\mathbb Z[p^{-1}:p\in P].
\end{equation}
The base \(B\) is noetherian, excellent, and regular, of dimension
at most one. Every positive residue characteristic of a
\(B\)-scheme belongs to \(P\). For a given \(S\) in
Theorem~\ref{thm:main}, take \(P\) to be its set of positive
residue characteristics. Every prime outside \(P\) is a unit in
each local ring of \(S\), and hence in its structure sheaf. Thus
\(S\) factors uniquely through \(B\), and the coefficient
hypothesis supplies a homomorphism \(\Lambda\to R\).

We first prove base change stability of \(H\Lambda\) on noetherian
finite-dimensional \(B\)-schemes and then extend scalars to \(R\).
The use of \(B\) ensures that finite-type
approximations introduce no new residue characteristics. The choices
\(P=\varnothing\), \(P=\{p\}\), and \(P\) equal to all primes
give, respectively,
\[
 (B,\Lambda)=
 (\operatorname{Spec}\mathbb Q,\mathbb Z),\quad
 (\operatorname{Spec}\mathbb Z_{(p)},\mathbb Z[1/p]),\quad
 (\operatorname{Spec}\mathbb Z,\mathbb Q).
\]

For a sheaf of spaces \(\mathcal P\) on \(\mathrm{Sm}_S\) and
a sheaf of \(\Lambda\)-modules \(\mathcal A\), put
\[
 M_S(\mathcal P)=\Sigma_T^\infty
     L_{\mathrm{Nis},\mathbb A^1}\Lambda[\mathcal P],
 \qquad
 N_S(\mathcal A)=\Sigma_T^\infty
     L_{\mathrm{Nis},\mathbb A^1}(\mathcal A[0]).
\]
The first construction is the free linear motive of a space. The
second uses its given additive structure. Both take values in
\(\mathrm{DA}(S,\Lambda)\). The functor \(M_S\) preserves
colimits and takes Cartesian products to tensor products.

Let \(\mathcal Z_{S,r}\) be the sheaf of integral universal finite
relative cycles on \(\mathbb G_{m,S}^r\), with transfers forgotten.
The reduced summand \(C_{S,r,\Lambda}\) and the normalized stages
of Section~\ref{sec:models} satisfy
\begin{equation}\label{aca-stages}
 E_{S,r,\Lambda}=N_S(C_{S,r,\Lambda})[-r](-r),
 \qquad
 \operatorname*{colim}_{r}E_{S,r,\Lambda}\simeq H\Lambda_S.
\end{equation}
Their canonical exchange maps, induced by relative-cycle pullback,
assemble to \(\tau_{f,\Lambda}\). The first case of interest is a
map between seminormal bases. On these bases the effective cycles
of degree \(n\), after inverting the residue characteristics as
specified in Section~\ref{sec:cycle-presentation}, are described by
the geometric symmetric power of the torus. We therefore first
study base change for each of these quotient schemes on the smooth site.

\section{Geometric base change for permutation quotients}
\label{acg-section}

Retain \(B\), \(P\), and \(\Lambda\) from
\eqref{aca-arithmetic-base}, and let \(f:T\to S\) be a morphism
of noetherian finite-dimensional \(B\)-schemes. For the symmetric
power of degree \(n\) the map we need to understand is
\[
 f^*M_S\bigl(h_{\operatorname{Sym}^n_S(\mathbb G_{m,S}^{r})}\bigr)
 \longrightarrow
 M_T\bigl(h_{\operatorname{Sym}^n_T(\mathbb G_{m,T}^{r})}\bigr).
\]
Its cofiber is \(\mathcal B_f((\mathbb G_m^r)^n,\Sigma_n)\)
in the notation of \eqref{acg-error}. We prove that this cofiber
vanishes over \(T[1/n!]\). Over a characteristic-\(p\) fiber we
prove that its telescope under the coordinate \(p\)-power map
vanishes. The stabilizer induction requires the larger class of
quotients defined below. Both arguments apply to ordinary Nisnevich
\(\mathrm{DA}\) with arbitrary coefficients and to bases without any
seminormality hypothesis.

A permutation model is an invariant open subscheme
\[
 Y\subseteq
 (\mathbb G_{m,\mathbb Z}^{r})^n\times_{\mathbb Z}
 \mathbb A_{\mathbb Z}^{d},
\]
with a finite constant group \(G\) acting by permutations of the
\(n\) torus blocks and trivially on the affine coordinates.
We also allow coordinate fixed loci, finite disjoint unions, and
induced schemes \(G\times_H V\).  These have the same coordinate
description componentwise, possibly with a constant kernel in the
permutation action.  All models and their morphisms in this section
are defined over \(\mathbb Z\), with restriction to open subsets
allowed.  In particular, adjoining the inverse of a nonzero integer
is allowed as an open restriction of a model.
Set \(Y_S=Y\times_{\mathbb Z}S\), and similarly for \(T\).

\subsection{Quotients, base change, and geometric errors}

\begin{lemma}\label{acg-quotient-base-change}
For every permutation model, formation of its geometric quotient
commutes with arbitrary base change from \(\operatorname{Spec}\mathbb Z\).
In particular there is a canonical isomorphism
\[
 (Y/G)\times_{\mathbb Z}S\simeq Y_S/G.
\]
The assertion includes the coordinate fixed loci, invariant opens,
finite disjoint unions, and induced schemes specified above.
\end{lemma}

\begin{proof}
For the full ambient affine model the coordinate ring is
\[
 A=\mathbb Z[x_{a,j}^{\pm1},z_1,\ldots,z_d].
\]
Its Laurent monomials form a \(\mathbb Z\)-basis permuted by \(G\).
An invariant element has constant coefficients on each monomial
orbit.  Consequently the orbit sums form a basis of \(A^G\), and,
for every ring \(R\), the canonical map
\[
 A^G\otimes_{\mathbb Z}R
       \longrightarrow (A\otimes_{\mathbb Z}R)^G
\]
is an isomorphism of rings.

The quotient morphism of the full ambient model is finite. An invariant
open \(Y\) is the inverse image of an open in that quotient, since
its invariant closed complement has closed image, and quotient
fibers are precisely the orbits.  Thus its quotient is this
open subscheme.  Covering the quotient open by principal opens
defined by invariant functions reduces its base-change assertion
to the affine calculation, since localization at an invariant
function commutes with invariants.  A coordinate fixed locus is
again a Laurent-polynomial model, and its normalizer permutes the
resulting orbit blocks.  The same calculation applies to it.
Finite disjoint unions and
\((G\times_H V)/G=V/H\) give the remaining assertions.
\end{proof}

For a possibly singular \(S\)-scheme \(Q\), write \(h_Q\) for the
presheaf \(U\mapsto\operatorname{Hom}_S(U,Q)\) on \(\mathrm{Sm}_S\).
Its expression as a colimit of representables is indexed by smooth
test schemes equipped with a map \(U\to Q\). Base changing these
maps gives \(U_T\to Q_T\), and thus a canonical comparison
\(f^*M_S(h_Q)\to M_T(h_{Q_T})\). For \(Q=Y_S/G\),
Lemma~\ref{acg-quotient-base-change} identifies \(Q_T\) with \(Y_T/G\).
Define the geometric base-change error by
\begin{equation}\label{acg-error}
 \mathcal B_f(Y,G)=
 \operatorname{cofib}\left(
 f^*M_S(h_{Y_S/G})\longrightarrow M_T(h_{Y_T/G})\right).
\end{equation}
This is a functor of equivariant maps, including maps which
increase stabilizers.  A constant kernel of the action can be
removed, since it acts trivially on both quotient presentations.

\begin{lemma}\label{acg-excision}
The errors \eqref{acg-error} have the following properties.
\begin{enumerate}
\item They satisfy Mayer--Vietoris for invariant open covers and
are invariant under adjoining an affine-line coordinate with
trivial group action.
\item An equivariant separated \'etale map \(V\to Y\) which preserves
geometric stabilizers and is an isomorphism over an invariant
closed subscheme \(Z\subseteq Y\) gives a pushout square of errors,
with the other two vertices the complements of \(Z\).
\item If \(Y/G\) is smooth over \(\mathbb Z\), the error is zero.
In particular this holds for a free action on a smooth model.
\item There are canonical identifications
\[
 \mathcal B_f(G\times_H V,G)\simeq\mathcal B_f(V,H),
\]
and errors of finite disjoint unions are the corresponding sums.
\end{enumerate}
\end{lemma}

\begin{proof}
Invariant open covers descend to open covers of the geometric
quotient.  Also
\((Y\times\mathbb A^1)/G=(Y/G)\times\mathbb A^1\).
These statements, the monoidal property of \(M\), and the
colimit properties of \(f^*\) prove the first assertion.

For the second, fixed-point-reflecting \'etale descent for finite
quotients gives a Cartesian square
\[
 \begin{matrix}
 V&\longrightarrow&Y\\
 \downarrow&&\downarrow\\
 V/G&\longrightarrow&Y/G
 \end{matrix}
\]
with \'etale bottom arrow by
\cite[Definitions~3.6 and~3.17, and Corollary~5.4]{RydhQuotients}.
The bottom arrow is an isomorphism over the reduced closed image
of \(Z\).  Indeed, after lifting a geometric point of that image
to \(Z\), the Cartesian square identifies its geometric fiber
with the singleton fiber upstairs. Over this closed subscheme the
\'etale map is therefore surjective and radicial, hence an
isomorphism. The quotient square with complementary opens
is therefore Nisnevich distinguished.

To check excision on the smooth site, pull the square back along a
smooth test scheme mapping to \(Y_S/G\). Its pullback is a Nisnevich
distinguished square of smooth \(S\)-schemes. Thus the square of
restricted representables is a pushout after
Nisnevich localization on the smooth site.  Apply the same argument
over \(T\), and then take the cofiber of the comparison of the
two squares.  In a stable category this is again a pushout.

For a smooth quotient the comparison in \eqref{acg-error} is
ordinary smooth representable base change.  A free finite constant
group action on a smooth model has an \'etale torsor as its quotient
map, so that quotient is smooth.  Finally, the identifications
for induction and disjoint unions hold already for the quotient
schemes and their canonical base-change arrows.
\end{proof}

\subsection{Stabilizer neighborhoods and the tame comparison}

For a disjoint union of models, separate the orbits of components.
On each such orbit the induction identity of Lemma~\ref{acg-excision}
allows us to work on a single component with its stabilizer group.
Remove a constant kernel if necessary. Let \(m\) be the largest order of a
geometric stabilizer on a nonempty model \(Y\).  For each subgroup
\(H\subseteq G\) of order \(m\), its fixed locus \(Y^H\) is smooth, since
it is an open in the coordinate diagonal which identifies all
torus blocks in each \(H\)-orbit.  Distinct nonempty such loci
are disjoint, since a point in their intersection would have
stabilizer of order greater than \(m\).  Their union
\[
 Z=\coprod_{|H|=m}Y^H
\]
is an invariant closed subscheme, and \(Y\setminus Z\) has smaller
stabilizers.

For a conjugacy representative \(H\), put \(N=N_G(H)\).
Remove from \(Y\) every other order-\(m\) fixed locus and the
fixed loci of every element of \(G\setminus N\).
The resulting \(N\)-invariant open \(V_H\) contains \(Y^H\).
All its stabilizers are contained in \(N\), and its maximal
fixed locus is \(Y^H\).  The map
\begin{equation}\label{acg-normalizer-neighborhood}
 \coprod_{[H]}G\times_{N_G(H)}V_H\longrightarrow Y
\end{equation}
is \'etale, preserves stabilizers, and is an isomorphism over \(Z\).
On each component it is an open immersion followed by a coordinate
permutation.  Lemma~\ref{acg-excision} therefore reduces an induction
on \(m\) to
\begin{equation}\label{acg-normalizer-case}
 H\triangleleft N,\qquad Z=V^H,
\end{equation}
where \(V\setminus Z\) has smaller stabilizers.
The group \(H\) acts trivially on \(Z\), and \(N/H\) acts freely.
Hence \(Z/N=Z/(N/H)\) is smooth and
\begin{equation}\label{acg-fixed-zero}
 \mathcal B_f(Z,N)=0.
\end{equation}

\begin{proposition}\label{acg-tame}
If \(n!\) is invertible on \(S\), then
\(\mathcal B_f(Y,G)=0\) for every displayed \(n\)-block model.
Without this hypothesis, the restriction of its error to
\(T[1/n!]\) is zero.
\end{proposition}

\begin{proof}
First replace the integer models by their opens on which \(n!\)
is invertible.  This changes neither base-changed term, and makes
the averaging formulas below maps of the permitted integer models.
Induct on the maximum stabilizer order, simultaneously for all
invariant opens and auxiliary affine coordinates.  The free case
is Lemma~\ref{acg-excision}.  Use
\eqref{acg-normalizer-neighborhood} and excision to reach
\eqref{acg-normalizer-case}.  For every \(H\)-orbit \(I\) of
torus blocks, replace each coordinate in that orbit by
\[
 R_{I,j}=|I|^{-1}\sum_{a\in I}x_{a,j}.
\]
These functions, with the affine coordinates unchanged, define
an \(N\)-equivariant retraction \(\rho:V'\to Z\) on an invariant
open \(V'\subseteq V\) containing all of \(Z\).
The cover \(V=V'\cup(V\setminus Z)\), together with induction
on its second member and intersection, shows that
\(\mathcal B_f(V',N)\to\mathcal B_f(V,N)\) is an equivalence.

Interpolate the torus coordinates between \(x_{a,j}\) and
\(R_{I,j}\).  Let \(W\subseteq V'\times\mathbb A^1\) be the
open where the interpolation lands in \(V\).  It contains both
endpoint fibers and \(Z\times\mathbb A^1\).  The cover
\[
 V'\times\mathbb A^1
   =W\cup((V'\setminus Z)\times\mathbb A^1)
\]
has smaller stabilizers on the second member and intersection.
Mayer--Vietoris and affine-line invariance identify the error of
\(W\) with that of \(V'\) by projection.  Both endpoint sections
are inverses to this projection and thus induce the same map.
Their composites with the interpolation identify the inclusion
\(V'\to V\) with a map factoring through \(Z\).
Equation~\eqref{acg-fixed-zero} makes this latter map zero.
The inclusion is an equivalence on errors, so the error is zero.

For general \(S\), restriction to the open \(T[1/n!]\) identifies
the error with that for
\(T[1/n!]\to S[1/n!]\).  Open pullback of a restricted
representable is the restricted representable of its open
base change.  The first assertion applies there.
\end{proof}

\subsection{Germs along a residue-characteristic fiber}

Fix a prime \(p\), write \(i_p:T_p\hookrightarrow T\) for the
closed fiber, and put
\[
 E_{f,p}(Y,G)=i_p^*\mathcal B_f(Y,G).
\]
If \(p\notin P\), this is the zero category calculation because
\(T_p\) is empty.  A \(p\)-neighborhood in a model \(Y\) is an
open containing its entire \(\mathbb F_p\)-fiber.

\begin{lemma}\label{acg-neighborhood}
An invariant \(p\)-neighborhood \(V\subseteq Y\) induces an
equivalence \(E_{f,p}(V,G)\simeq E_{f,p}(Y,G)\).
This remains true with any number of auxiliary affine parameters.
\end{lemma}

\begin{proof}
There is an invariant open cover \(Y=V\cup Y[1/p]\).
For an arbitrary scheme \(Q\) over \(S[1/p]\), even if \(Q\)
is singular, extension by zero on the smooth site gives
\[
 M_S(h_Q)\simeq
 j_{S\sharp}M_{S[1/p]}(h_Q),\qquad
 j_S:S[1/p]\hookrightarrow S.
\]
One can verify this before localization by left Kan extension
of representables, and then pass through free linearization.
Open base change gives
\(f^*j_{S\sharp}\simeq j_{T\sharp}f_{[1/p]}^*\).
Both terms of the error for \(Y[1/p]\), and for its intersection
with \(V\), consequently restrict to zero under \(i_p^*\),
since \(i_p^*j_{T\sharp}=0\).
Mayer--Vietoris proves the assertion.  The argument is unchanged
after adjoining affine parameters. The localization theorem in
ordinary \(\mathrm{DA}\) used here holds with arbitrary coefficients
\cite[Theorem~6.2.1]{CD}.
\end{proof}

A map on a germ is a map defined on a \(p\)-neighborhood, with
two maps identified after restriction to a common such neighborhood.
Intersections give common refinements, and inverse images allow
composition.  Lemma~\ref{acg-neighborhood} permits application of
\(E_{f,p}\) to these germs, or equivalently to the localization
at their open inclusions.

Coordinatewise \(p\)-th powers on all torus and affine coordinates
give a germ
\[
 P_Y:Y\dashrightarrow Y.
\]
Its domain \(Y\cap P^{-1}(Y)\) is a \(p\)-neighborhood, because
its reduction is absolute Frobenius of the \(\mathbb F_p\)-model.
For a finite disjoint union this operation fixes the component.
It induces an endomorphism \(\psi_Y\) of \(E_{f,p}(Y,G)\).
Coordinate permutations, fixed diagonals, open inclusions,
projections, induction on components, and sections at \(0,1\)
commute with these germs.  All affine parameters are raised to
their \(p\)-th powers as part of \(P_Y\).

\begin{lemma}\label{acg-lifts}
Two equivariant germs \(a,b:Y\dashrightarrow Z\) with the same
reduction modulo \(p\) induce homotopic maps under \(E_{f,p}\).
Their coordinate interpolation supplies a specified such homotopy.
In particular it supplies
\[
 E_{f,p}(P_Za)\simeq E_{f,p}(aP_Y).
\]
The homotopy is constant on a parameter face where the two
coordinate formulas agree exactly.  It factors through a fixed
coordinate diagonal on a face where both formulas do so.
\end{lemma}

\begin{proof}
Restrict to a common domain.  If necessary, further restrict to
the open-and-closed \(p\)-neighborhood where the two maps land
in the same target component.  Interpolate each ambient coordinate
by \((1-s)a+s b\).  The open where the torus coordinates are
units and the image lies in \(Z\) contains the entire
characteristic-\(p\) fiber times \(\mathbb A^1\), since on that fiber
the interpolation is independent of \(s\).
Lemma~\ref{acg-neighborhood} and affine-line invariance therefore
turn it into the asserted homotopy.  Equal coordinate formulas
remain equal throughout this interpolation, which gives the
relative statements.  Finally, \(P_Za\) and \(aP_Y\) have the
same reduction by naturality of absolute Frobenius over
\(\mathbb F_p\).
\end{proof}

For an object with endomorphism \(\psi\), use the notation
\[
 \operatorname{Tel}_{\psi}(E)=
 \operatorname*{colim}
       (E\xrightarrow{\psi}E\xrightarrow{\psi}E\longrightarrow\cdots).
\]
The induced endomorphism on this telescope is invertible.
Colimits of diagrams with commuting endomorphisms commute
with this construction.

\begin{lemma}\label{acg-sequential}
A map \(a:E\to F\) with a specified homotopy
\(a\psi_E\simeq\psi_Fa\) induces a map of these telescopes.
If the map and its specified square factor through a third such
diagram, the telescope map has the same factorization.
The construction also applies to a finite diagram equipped
with compatible specified squares.
\end{lemma}

\begin{proof}
The specified square is an endomorphism of the object
\(a\) in the arrow infinity-category.  Iterating it gives a
sequence in that category.  A string of composable arrows extends
to the free sequential indexing category with a contractible
space of choices, so this specifies the required sequential
diagram.  Its colimit is computed on source and target separately.
For a factorization, or a finite diagram, use its functor category
instead of the arrow category.  This proves all assertions.
\end{proof}

\subsection{Arithmetic stabilizer induction}

\begin{proposition}\label{acg-arithmetic}
For every permutation model and every prime \(p\),
\[
 \operatorname{Tel}_{\psi}
       i_p^*\mathcal B_f(Y,G)=0.
\]
For the full torus model, \(\psi\) is induced by the actual
coordinate \(p\)-power map on its geometric quotient.  In
particular, for \(G=\Sigma_n\) the conclusion is the vanishing
of the power telescope of the canonical geometric symmetric-power
base-change error.
\end{proposition}

\begin{proof}
Remove a constant kernel of the action, which leaves the quotient
and error unchanged.  Induct on the maximum stabilizer order,
simultaneously for all invariant opens and auxiliary affine
coordinates. Lemma~\ref{acg-excision} settles the free case before
taking a telescope. The empty model also has zero error.

The stabilizer neighborhoods
\eqref{acg-normalizer-neighborhood} are defined over
\(\mathbb Z\) and their structural maps commute with power germs.
Excision, induction on the complements, and the compatibility
with induction of group actions reduce the problem to
\eqref{acg-normalizer-case}.  Here
\(E_{f,p}(Z,N)=0\) by \eqref{acg-fixed-zero}.

For an \(H\)-orbit \(I\) of torus blocks put
\[
 d_I=|I|,\qquad a_I=p^{v_p(d_I)},\qquad
 b_I=\binom{d_I}{a_I}.
\]
To see that \(b_I\) is prime to \(p\), write \(d_I=a_Id'_I\)
with \(p\nmid d'_I\). The coefficient of \(t^{a_I}\) in
\((1+t)^{d_I}=(1+t^{a_I})^{d'_I}\) modulo \(p\) is \(d'_I\).
Inverting these finitely many integers restricts the model to
a \(p\)-neighborhood and does not change \(E_{f,p}\).

Choose \(q=p^e\) divisible by every \(a_I\), and put
\[
 R_{I,j}(x)=
 \left(b_I^{-1}
       e_{a_I}(x_{a,j}:a\in I)\right)^{q/a_I}.
\]
On the coordinate diagonal fixed by \(H\) there is the exact
identity \(R_{I,j}=x_j^q\), over the localized integer ring.
Moreover \(N\) permutes the \(H\)-orbits of a given size, so
these formulas are \(N\)-equivariant.  On each auxiliary affine
coordinate use its \(q\)-th power.

Restrict to the invariant open \(V'\subseteq V\) on which the
displayed torus coordinates are units, their image lies in \(Z\),
and the coordinate \(q\)-power map lands in \(V\).
This defines \(\rho:V'\to Z\).  The open \(V'\) contains
\(Z_{\mathbb F_p}\), because on \(Z\) the formulas are the
\(q\)-power germ and modulo \(p\) are Frobenius.
It follows that \(V'\cup(V\setminus Z)\) is a \(p\)-neighborhood
of \(V\).  Induction on \(V\setminus Z\) and on its intersection
with \(V'\), together with Mayer--Vietoris and
Lemma~\ref{acg-neighborhood}, proves that the inclusion \(V'\to V\)
is an equivalence on power-telescoped errors.

Consider the interpolation
\[
 \mathcal H_\lambda(x_{a,j})=
       (1-\lambda)x_{a,j}^{q}+\lambda R_{I,j}(x),
       \qquad a\in I.
\]
Keep the already \(q\)-powered affine coordinates fixed.
Let \(W\subseteq V'\times\mathbb A^1\) be the open where this
defines a map to \(V\).  It contains both complete endpoint
fibers and \((V'\cap Z)\times\mathbb A^1\), since on that
fixed locus the interpolation is constant.  Hence
\[
 V'\times\mathbb A^1
   =W\cup((V'\setminus Z)\times\mathbb A^1).
\]
The second member and intersection have smaller stabilizers.
All inclusions and the projection commute with power germs.
Taking telescopes in Mayer--Vietoris, and using affine-line
invariance, gives an equivalence
\begin{equation}\label{acg-projection}
 \operatorname{Tel}_{\psi}E_{f,p}(W,N)
       \longrightarrow
 \operatorname{Tel}_{\psi}E_{f,p}(V',N)
\end{equation}
induced by projection.  Its two endpoint sections are both
inverses and therefore agree.

To obtain a map of telescopes from the interpolation we must provide
a homotopy commuting it with the power maps. Apply Lemma~\ref{acg-lifts} to the germs
\(P_V\mathcal H\) and \(\mathcal H P_W\).
Their reductions agree, with \(P_W\) also raising \(\lambda\)
to its \(p\)-th power.  Their coordinate interpolation supplies
the square required by Lemma~\ref{acg-sequential}.
At \(\lambda=0\) the formulas agree exactly, since both
are coordinate \(pq\)-th powers.  The specified square is
therefore constant on that endpoint.
At \(\lambda=1\) both formulas, and their interpolation,
land in \(Z\). Thus the specified square is the inclusion of the
corresponding square for \(\rho:V'\to Z\).
These relative statements give compatible squares for the
finite diagram consisting of \(\mathcal H\), its endpoint
sections, and this factorization through \(Z\).

Consequently the induced telescope map has, at its zero
endpoint, the canonical \(q\)-power germ \(V'\dashrightarrow V\).
At its other endpoint it factors through
\(\operatorname{Tel}_{\psi}E_{f,p}(Z,N)=0\).
The endpoint sections in \eqref{acg-projection} agree, so the
\(q\)-power map on these telescopes is zero. It is also an
equivalence, since it is the inclusion \(V'\to V\), already proved
an equivalence on telescopes, composed with the invertible
endomorphism \(\psi^e\).  Thus both telescoped errors are zero.
This proves the induction.
\end{proof}

For each fixed degree the error vanishes on the tame open
\(T[1/n!]\), and its restriction to every characteristic-\(p\)
fiber vanishes after the indicated power telescope. In the next section
the coordinate power action on the reduced cycle summand identifies
this telescope calculation with ordinary base change after inverting
the residue characteristics.

\section{From symmetric powers to additive cycle spectra}
\label{sec:cycle-presentation}

Let \(B\), \(P\), and \(\Lambda\) be as in
\eqref{aca-arithmetic-base}. We prove base change for the finite-cycle
spectrum on seminormal bases. At a fixed torus rank \(r\), the geometric
input concerns effective families of each symmetric degree \(n\).
To recover the \(r\)-th Tate stage we pass through group completion,
the additive motive of the resulting sheaf, and its fully reduced
coordinate summand. We describe these operations before assembling
the stages.

\subsection{Symmetric powers and universal relative cycles}

For a noetherian \(B\)-scheme \(S\), put
\(X_{S,r}=\mathbb G_{m,S}^r\) and let
\[
 D_{S,r}=\coprod_{n\geq0}h_{\operatorname{Sym}_S^n X_{S,r}}
\]
on its smooth Nisnevich site. Addition of families makes \(D_{S,r}\)
a commutative monoid sheaf. The sheaf coproduct allows the degree to
vary on open-and-closed components of a test scheme.
Write \(C^+_{S,r}\) for the sheaf of effective integral universal
finite relative cycles on \(X_{S,r}\), and \(\mathcal Z_{S,r}\) for the
integral finite-correspondence sheaf represented by \(X_{S,r}\), with
transfers forgotten. For any commutative monoid sheaf \(D\), write
\(D[P^{-1}]\) for the filtered colimit of the maps multiplying
multiplicities by elements of the multiplicative monoid generated by
\(P\). For example, the multiplicity map indexed by \(p\) sends
a family of degree \(n\) to the sum of \(p\) copies of it, of degree
\(pn\). This operation takes place in sheaves before applying the
free linear motive.

\begin{lemma}[Uniform cycle denominators]\label{lem:acc-denominators}
Suppose that \(S\) is seminormal. There is a canonical inclusion
\(C^+_{S,r}\hookrightarrow D_{S,r}\), inducing an isomorphism
\begin{equation}\label{eq:acc-positive-localization}
 C^+_{S,r}[P^{-1}]\xrightarrow{\ \sim\ }D_{S,r}[P^{-1}].
\end{equation}
For a family of constant degree \(n\geq1\), one denominator that
works simultaneously at every point is
\begin{equation}\label{eq:acc-denominator-bound}
 m_{P,n,r}=
 \prod_{\substack{p\in P\\p\leq n^r}}
       p^{\lfloor\log_p(n^r)\rfloor}.
\end{equation}
This is the \(P\)-primary part of
\(\operatorname{lcm}(1,\ldots,n^r)\), with empty product one.
Multiplication by this integer makes the family an integral universal
cycle. The identifications respect base change, addition, and finite
coordinate power maps.
\end{lemma}
\begin{proof}
Every smooth test \(U/S\) is seminormal, since seminormality ascends
under regular morphisms \cite[Proposition~38]{Kollar}. Rydh's
Theorem~10.17 in \cite{RydhCycles} identifies families with effective
rational cycles satisfying the limit-cycle condition and having
quasi-integral specializations at every point. Integral universal
cycles satisfy these conditions. Families of degree \(n\) with finite
support are represented by \(\Gamma_U^n(X_{U,r})\), which equals
\(\operatorname{Sym}_U^n(X_{U,r})\) because \(X_{U,r}/U\) is flat
\cite[Corollary~4.2.5]{RydhZero}. This gives the inclusion in the
statement.

Fix a point \(u\in U\) and a support point of a degree-\(n\)
family over it, with residue field \(L/\kappa(u)\). Each of its
\(r\) torus coordinates satisfies the monic degree-\(n\) polynomial
obtained by projecting the symmetric family to that coordinate.
Those coordinates generate \(L\), so
\[
 [L:\kappa(u)]\leq n^r.
\]
By \cite[Definitions~8.4--8.5 and Theorem~8.6]{RydhCycles}, a
coefficient denominator divides the inseparable discrepancy, which
divides the inseparable degree. In characteristic zero this denominator
is one. In characteristic \(p>0\), it is a \(p\)-power at most
\(n^r\), and \(p\in P\). Hence it divides \(m_{P,n,r}\).
Only finitely many primes occur in \eqref{eq:acc-denominator-bound},
even when \(P\) is infinite.

Multiplication by \(m_{P,n,r}\) preserves the limit-cycle condition
and makes every point specialization integral. Definition~8.1.47 and
Lemma~8.1.48 of \cite{CD} identify this pointwise condition with
integrality after arbitrary base change. The resulting cycle is
integral and universal. The assertion in degree zero is immediate,
and locally constant degrees may be treated on their open-and-closed
pieces. This proves surjectivity in \eqref{eq:acc-positive-localization}.
Uniqueness of the underlying cycles in Rydh's theorem and cancellation
give injectivity. The same uniqueness and the functoriality of relative
cycles give the stated compatibilities.
\end{proof}

\begin{lemma}[Group completion]\label{lem:acc-group-completion}
For seminormal noetherian \(S/B\), there is a canonical isomorphism
\begin{equation}\label{eq:acc-group-completion}
 a_{\mathrm{Nis}}\bigl(D_{S,r}[P^{-1}]^{\mathrm{gp}}\bigr)
 \xrightarrow{\ \sim\ }
 \mathcal Z_{S,r}\otimes_{\mathbb Z}\Lambda.
\end{equation}
The right side is also the finite-correspondence sheaf with
\(\Lambda\)-coefficients.
\end{lemma}
\begin{proof}
First show that the sheafified group completion of \(C^+_{S,r}\)
is \(\mathcal Z_{S,r}\). Injectivity follows from cancellation. For
surjectivity, work at a Nisnevich stalk, over the henselization \(U\)
of a smooth test at a point. Let \(\alpha\) be an integral universal
finite relative cycle, and let \(W\) be the reduced union of the
closures of its negatively weighted generic components. It is finite
over \(U\). If it is empty, \(\alpha\) is already effective.

In the ideal of \(W\) on \(X_{U,r}\), choose \(r\) functions whose
restrictions form a system of parameters in the smooth closed fibre
at each point of the closed fibre of \(W\). Semilocal prime avoidance
supplies the functions, and denominators can be cleared while preserving
their parameter property at these finitely many points. Their images
form regular sequences in the regular local rings of that fibre. The
fibrewise regular-sequence criterion and openness of flatness give an
open part \(V_0\) of their common zero scheme which is flat and
quasi-finite over \(U\) and contains those points
\cite[Tags~00MG and~00RC]{Stacks}. It contains all of \(W\), since
the omitted closed part of \(W\) is finite over the local base and
has empty closed fibre.

Henselian quasi-finite decomposition gives an open-and-closed part
\(V\subseteq V_0\), finite over \(U\), containing the whole closed
fibre \cite[Tag~02LN]{Stacks}. The same finite-support argument
shows that \(W\subseteq V\). Thus \(V/U\) is finite locally free.
Its map to the separated \(U\)-scheme \(X_{U,r}\) is a proper
immersion, hence a closed immersion. The fundamental relative cycle
\(\beta=[V/U]\) is effective and integral universal, and every
negatively weighted generic component of \(\alpha\) occurs in
\(\beta\) with positive multiplicity. For a sufficiently large
integer \(a\), the expression
\[
 \alpha=(\alpha+a\beta)-a\beta
\]
writes \(\alpha\) as a difference of effective integral universal
cycles. For \(r=0\) the construction uses the empty regular sequence.
This proves stalkwise surjectivity.

Group completion, sheafification, and localization commute. Applying
Lemma~\ref{lem:acc-denominators} now proves
\eqref{eq:acc-group-completion}. The identification with
finite-correspondence sheaves for these coefficients also follows
from \cite[Remark~9.1.3(3)]{CD}.
\end{proof}

\subsection{Power maps and additive motives}

\begin{lemma}[Power maps on additive cycle objects]
\label{lem:acc-power-action}
Let \(S\) be any noetherian scheme, let \(N\geq1\) be an integer,
and apply effective Nisnevich \(\mathbb A^1\)-localization to the
additive complex \(\mathcal Z_{S,r}[0]\). For each coordinate,
let \(c_j\) denote the endomorphism which replaces that coordinate
by \(1\). The commuting idempotents
\[
 e_I=\prod_{j\in I}(1-c_j)\prod_{j\notin I}c_j,
 \qquad I\subseteq\{1,\ldots,r\},
\]
split the additive object. Coordinatewise \(N\)-th power acts by
\(N^{|I|}\) on the summand selected by \(e_I\). In particular,
it acts by \(N^r\) on the fully reduced summand. These assertions
hold integrally, are natural in \(S\), and survive scalar extension
and Tate stabilization.
\end{lemma}
\begin{proof}
For one coordinate use the polynomial over \(\mathbb Z\)
\[
 P_t(z)=(1-t)(z-x)^N+t(z-x^N)(z-1)^{N-1}.
\]
It is monic of degree \(N\) with unit constant coefficient
\((-1)^N x^N\). Its zero scheme is consequently finite locally
free of degree \(N\) over \(\mathbb G_m\times\mathbb A^1\)
and lies in the target \(\mathbb G_m\). It defines an integral
universal finite correspondence. Its endpoint cycles are
\(N\Gamma_{\mathrm{id}}\) and
\(\Gamma_{[N]}+(N-1)\Gamma_1\). Composing with any universal
finite cycle supplies an additive \(\mathbb A^1\)-homotopy.
Consequently
\[
 [N]_* = N\,\mathrm{id}+(1-N)c_1.
\]
The homotopy is constant on the unit section, where its polynomial
is \((z-1)^N\).

For several coordinates take the product of this correspondence
with the identity correspondences in the remaining coordinates.
The coordinate operations commute. On the \(e_I\)-summand their
eigenvalues are \(N\) for \(j\in I\) and \(1\) otherwise, so
their composite has eigenvalue \(N^{|I|}\). The computation takes
place on the additive cycle sheaf itself and is preserved by the
stated functors.
\end{proof}

\begin{lemma}[Additive linearization after restriction]
\label{aca-linearization}
Let \(f:T\to S\) and \(i:V\to T\) be morphisms, and let
\(\mathcal A_S,\mathcal A_T\) be sheaves of \(\Lambda\)-modules
with a specified additive base-change map. If the induced map on free
linear motives of their underlying spaces,
\[
 i^*f^*M_S(\mathcal A_S)\longrightarrow i^*M_T(\mathcal A_T),
\]
is an equivalence, then the induced map on their additive motives,
\[
 i^*f^*N_S(\mathcal A_S)\longrightarrow i^*N_T(\mathcal A_T),
\]
is an equivalence. The same conclusion holds after a sequential
telescope of compatible additive endomorphisms if the
underlying-space comparison holds for the telescoped sheaves.
\end{lemma}
\begin{proof}
Apply the hypothesis to the iterated bar constructions for the additive
group laws. Their simplicial terms are finite Cartesian powers, which
the free linear motive takes to tensor powers. Since the motivic
pullbacks are symmetric monoidal and all three functors preserve
geometric realizations, the comparison is an equivalence on the
Eilenberg--MacLane spaces \(K(\mathcal A,j)\) for every \(j\geq0\).
The zero sections commute with the comparison, giving the same result
for reduced motives.

Write \(v\) for the conservative functor from ordinary linear motives
to motivic spectra which forgets the
\(H\Lambda^{\mathrm{top}}\)-module structure.
Simplicial stabilization gives the natural equivalence
\[
 v\left(\operatorname*{colim}_{j}
       \widetilde M(K(\mathcal A,j))[-j]\right)
 \simeq H\Lambda^{\mathrm{top}}\otimes vN(\mathcal A).
\]
Indeed, before motivic localization the stabilization of the
Eilenberg--MacLane spaces represents the abelian sheaf in degree
zero, while free \(\Lambda\)-linearization tensors its underlying
spectrum with \(H\Lambda^{\mathrm{top}}\). Motivic localization,
Tate stabilization, and the pullbacks preserve the displayed
constructions. Thus, for the additive comparison \(u\) over \(V\),
the bar constructions give an equivalence
\(H\Lambda^{\mathrm{top}}\otimes v(u)\).

For any \(H\Lambda^{\mathrm{top}}\)-module with underlying object
\(E\), the maps
\[
 E\xrightarrow{\eta\otimes 1}
 H\Lambda^{\mathrm{top}}\otimes E
 \xrightarrow{\mathrm{act}}E
\]
compose to the identity by unitality. Both maps are natural in module
morphisms. Applied to the source and target of \(u\), they exhibit
\(v(u)\) as a retract of \(H\Lambda^{\mathrm{top}}\otimes v(u)\)
in the arrow category. A retract of an equivalence is an equivalence,
and \(v\) is conservative. For the last assertion apply the same
argument to the sequential colimits of the sheaves. Filtered colimits
preserve their additive structure, commute with finite Cartesian
products, and pass through the additive motive.
\end{proof}

\subsection{Base change on seminormal bases}

\begin{remark}[Coordinate operations and Tate bonding maps]
\label{rem:acc-functoriality}
The comparisons in \eqref{eq:acc-group-completion} are induced by
the relative-cycle maps and respect addition, exterior products,
coordinate insertion and projection, and pullback between seminormal
bases. In particular, the coordinate projectors act on the comparison
arrows. At fixed \(r\), the pushforward \(u_p\) under coordinatewise
\(p\)-th power raises the torus coordinates of each point to their
\(p\)-th powers and preserves the symmetric degree. Multiplication of multiplicities by
\(p\) takes degree \(n\) to degree \(pn\). These two
operations commute. Increasing \(r\) uses the Tate bonding map,
defined by exterior product with the graph of a new coordinate.
Naturality for this product gives
\[
 \beta_{f,r+1}\,f^*b_{S,r}=b_{T,r}\,\beta_{f,r}
\]
for the canonical cycle exchanges. We will use the power maps at each
fixed \(r\) and use this bonding identity to pass to the colimit in
\(r\).
\end{remark}

\begin{proposition}[Canonical stages on seminormal bases]
\label{aca-seminormal-stages}
Let \(f:T\to S\) be a morphism of seminormal noetherian
finite-dimensional \(B\)-schemes. For every \(r\geq0\), the
canonical maps
\[
 f^*N_S(\mathcal Z_{S,r}\otimes\Lambda)
       \longrightarrow N_T(\mathcal Z_{T,r}\otimes\Lambda),
 \qquad
 f^*E_{S,r,\Lambda}\longrightarrow E_{T,r,\Lambda}
\]
are equivalences. They respect the coordinate projectors, external
products, and Tate bonding maps. Consequently \(\tau_{f,\Lambda}\)
is an equivalence of commutative algebras.
\end{proposition}
\begin{proof}
Fix \(r\) and let \(D_S=D_{S,r}\), \(D_T=D_{T,r}\).
Write \(K_r\) for the cofiber of the first displayed map.
For \(U=S,T\), let \(D_U^{(n)}=h_{\operatorname{Sym}^n_U X_{U,r}}\).
Since the symmetric power is the geometric quotient of
\(X_{U,r}^n\) by \(\Sigma_n\), the definition \eqref{acg-error}
specializes to
\begin{equation}\label{eq:cycle-degree-error}
 \operatorname{cofib}\bigl(f^*M_S(D_S^{(n)})\longrightarrow
                                  M_T(D_T^{(n)})\bigr)
 \simeq \mathcal B_f\bigl((\mathbb G_{m,\mathbb Z}^r)^n,\Sigma_n\bigr).
\end{equation}
The map on the left is the geometric base-change map on restricted
representables, followed by the free linear motive. It is compatible
with addition of symmetric families. We describe how its comparisons
in the separate degrees induce the canonical additive comparison whose
cofiber is \(K_r\).

\smallskip\noindent\emph{Characteristic zero.}
First take \(B=\operatorname{Spec}\mathbb Q\), so \(P=\varnothing\)
and \(\Lambda=\mathbb Z\). Proposition~\ref{acg-tame} applies to
every \(n\), since every \(n!\) is invertible on \(S\). Consequently
\eqref{eq:cycle-degree-error} vanishes for each \(n\). Taking the
coproduct over all symmetric degrees gives an equivalence
\(f^*M_S(D_S)\to M_T(D_T)\).

To pass from effective families to their differences, consider a
cancellative commutative monoid \(D\). Its translation category has
objects \((a,b)\in D^2\) and arrows
\((a,b)\to(a+c,b+c)\) indexed by \(c\in D\). Its nerve in
degree \(q\) is \(D^{q+2}\). Each component has constant difference
\(a-b\in D^{\mathrm{gp}}\), and all pairs with this difference lie
in a single filtered component. Indeed, if \(a-b=a'-b'\), translating
the first pair by \(b'\) and the second by \(b\) gives a common
target. Cancellation gives uniqueness of parallel arrows. The nerve
therefore realizes to \(D^{\mathrm{gp}}\). This argument applied to
Nisnevich stalks identifies the sheafified realization with the
sheafified group completion. The monoids in this application are
cancellative by Lemma~\ref{lem:acc-denominators}.

The equivalence for \(D_S,D_T\) passes to each term of these nerves because
\(M\) sends finite Cartesian products to tensor products, and \(f^*\)
is symmetric monoidal. It passes to the realizations because both
functors preserve colimits. Using \eqref{eq:acc-group-completion},
we obtain the equivalence on free motives of the underlying spaces
of \(\mathcal Z_{S,r},\mathcal Z_{T,r}\). Lemma~\ref{aca-linearization}
with \(i=\mathrm{id}_T\) gives \(K_r=0\). This proves the additive
assertion integrally in characteristic zero.

\smallskip\noindent\emph{The arithmetic case.}
Return to arbitrary \(P\). Let \(i_{\mathbb Q}:T_{\mathbb Q}\to T\)
be the characteristic-zero fibre. For each fixed \(n\),
Proposition~\ref{acg-tame} gives the vanishing of
\eqref{eq:cycle-degree-error} on \(T[1/n!]\), so its pullback
along \(i_{\mathbb Q}\) is zero. After taking the coproduct over
\(n\), localize the monoids at the multiplicities in \(P\).
This is a filtered colimit of maps respecting the comparisons, so it
preserves the equivalence after pullback. By
Lemma~\ref{lem:acc-denominators} the localized monoids are
cancellative. The translation nerves and \eqref{eq:acc-group-completion}
then apply exactly as above, followed by Lemma~\ref{aca-linearization}
with \(i=i_{\mathbb Q}\). This proves \(i_{\mathbb Q}^*K_r=0\).

Next let \(p\in P\) and \(i_p:T_p\to T\) be the corresponding
closed fibre. In each symmetric degree, Proposition~\ref{acg-arithmetic}
gives vanishing of the \(u_p\)-telescope of the
\(i_p^*\)-pullback of \eqref{eq:cycle-degree-error}. Thus on this
fibre the geometric comparison is an equivalence after the telescope.
The construction respects coproducts in the symmetric degree and
commutes with the multiplicity maps. We may therefore make the
multiplicity localization and form the telescope on the monoids.
The sequential colimit commutes with each finite Cartesian power in
the translation nerves, since the diagonal
\(\mathbb N\to\mathbb N^{m}\) is cofinal. Filtered colimits
preserve cancellation and commute with group completion. Taking the
realizations yields the equivalence on the free motives of the
telescoped group completions after pullback. The telescope assertion
of Lemma~\ref{aca-linearization} gives
\begin{equation}\label{aca-additive-telescope}
 \operatorname*{colim}_{u_p}i_p^*K_r=0.
\end{equation}
Here \(i_p^*K_r\) is the pullback of the cofiber formed over \(T\).
Indeed, in each symmetric degree we applied \(i_p^*\) to the
comparison arrow of \eqref{eq:cycle-degree-error} itself. All subsequent
group completions and additive comparisons concern that same arrow
after restriction to the fibre.

On the additive motives of both cycle sheaves,
Lemma~\ref{lem:acc-power-action} computes \(u_p\) as multiplication
by \(p^{|I|}\) on the summand indexed by \(I\). Each scalar is a
unit in \(\Lambda\). The compatible action on the arrow defining
\(K_r\) is therefore an equivalence, as is its action on
\(i_p^*K_r\). The map from the initial term of a sequential
diagram of equivalences to its colimit is an equivalence.
Equation~\eqref{aca-additive-telescope} consequently gives
\(i_p^*K_r=0\).

Every point of \(T\) factors through \(T_{\mathbb Q}\) or one
of the \(T_p\). Pullbacks to residue fields are jointly conservative
in \(\mathrm{DA}(T,\Lambda)\), including on noncompact objects, by
the conservative ordinary linear forgetful functor and
\cite[Proposition~3.2]{BEM}, or directly by
\cite[Proposition~4.3.17 and Example~6.1.13]{CD}. Hence \(K_r=0\).
This completes the proof of the first assertion for all \(P\).

\smallskip\noindent\emph{Reduction and the Tate colimit.}
By Remark~\ref{rem:acc-functoriality} the additive comparison
commutes with all coordinate resets \(c_j\). Its retract cut out by
\(\prod_j(1-c_j)\) is therefore an equivalence as well. Applying
the shift and twist in \eqref{aca-stages} gives
\(\beta_{f,r}:f^*E_{S,r,\Lambda}\xrightarrow{\sim}E_{T,r,\Lambda}\).
These are the canonical cycle maps, so their compatibility with
exterior products and the graph of the new coordinate makes them a
map of the Tate diagrams. We may now vary \(r\). Since \(f^*\)
preserves colimits, Proposition~\ref{prop:cycle-tower-presentation}
identifies their colimit with
\[
 f^*H\Lambda_S\simeq
 \operatorname*{colim}_{r} f^*E_{S,r,\Lambda}
 \xrightarrow{\ \sim\ }\operatorname*{colim}_{r} E_{T,r,\Lambda}
 \simeq H\Lambda_T .
\]
This colimit is \(\tau_{f,\Lambda}\), and an equivalence on underlying
objects detects an equivalence of commutative algebras.
\end{proof}

\section{Base change stability of \texorpdfstring{\(HR\)}{HR} on general bases}
\label{sec:general-exchange}

Fix the arithmetic base \(B\), the set of primes \(P\), and the ring
\(\Lambda\) of~\eqref{aca-arithmetic-base}. For a \(B\)-scheme
\(S\), write \(p_S:S\to B\) for the structural morphism. We first
prove that
\[
 \tau_{p_S,\Lambda}:p_S^*H\Lambda_B\longrightarrow H\Lambda_S
\]
is an equivalence. Proposition~\ref{aca-seminormal-stages} gives
this assertion when \(S\) is seminormal. We descend it along
finite seminormalization for schemes of finite type over \(B\), and
then use affine approximation for general bases. The approximation
will use continuity in both \(\mathrm{DA}\) and \(\mathrm{DM}\).

\subsection{Descent along finite seminormalization}

For a smooth morphism \(a:V\to U\), the exchange
\(a^*G_{U,\Lambda}\to G_{V,\Lambda}a^*\) is an equivalence.
Indeed, its left adjoint comparison identifies the two composites
of smooth pushforward \(a_\sharp\) and extension of transfers.
These composites agree on smooth representables by composition
of smooth morphisms and hence agree on the categories they
generate. In particular, \(\tau_{a,\Lambda}\) is an equivalence
for smooth \(a\), including open immersions. We use this case
of exchange independently of the assertion for arbitrary morphisms.

\begin{lemma}[Finite universal homeomorphisms]
\label{aca-homeomorphisms}
Let \(q:S'\to S\) be a finite universal homeomorphism of noetherian
finite-dimensional \(B\)-schemes. Then
\(q^*:\mathrm{DA}(S,\Lambda)\to\mathrm{DA}(S',\Lambda)\) is an
equivalence, and \(\tau_{q,\Lambda}\) is an equivalence.
\end{lemma}

\begin{proof}
The residue characteristic exponents of \(S\) are units in
\(\Lambda\). Topological invariance
\cite[Theorem~2.1.1]{EK} therefore makes \(q^*\) an equivalence
from \(\mathrm{SH}(S)[P^{-1}]\) to \(\mathrm{SH}(S')[P^{-1}]\).
It is symmetric monoidal and takes the ordinary coefficient algebra
\(H\Lambda^{\mathrm{top}}_S\) to \(H\Lambda^{\mathrm{top}}_{S'}\).
Passing to modules and using Lemma~\ref{lem:linear-model} proves
the first assertion.

Consider the adjoint \(H\Lambda_S\to q_*H\Lambda_{S'}\) of the
exchange map. Test it against the smooth Tate objects
\(a_\sharp\mathbf1_{X,\Lambda}(-m)[-n]\), with
\(a:X\to S\) smooth separated of finite type and \(m,n\in\mathbb Z\).
Smooth exchange and adjunction identify the induced homomorphism
with pullback in motivic cohomology
\[
 H^{n,m}_{\mathrm M}(X,\Lambda)\longrightarrow
 H^{n,m}_{\mathrm M}(X\times_S S',\Lambda).
\]
The projection \(X\times_S S'\to X\) is a finite universal
homeomorphism. The displayed map is an isomorphism by
\cite[Proposition~11.2.7]{CD}. The smooth Tate objects and their
shifts generate \(\mathrm{DA}(S,\Lambda)\), so the adjoint map is an
equivalence. Since \(q^*\dashv q_*\) is an adjoint equivalence, this
also proves that \(\tau_{q,\Lambda}\) is an equivalence.
\end{proof}

\begin{proposition}[Structural exchange on finite-type bases]
\label{aca-finite-type}
For each finite-type \(B\)-scheme \(S\), the canonical map
\(\tau_{p_S,\Lambda}:p_S^*H\Lambda_B\to H\Lambda_S\) is an
equivalence.
\end{proposition}

\begin{proof}
Write \(q:S^{\mathrm{sn}}\to S\) for the seminormalization of the
reduction followed by the reduction map. Since \(B\) is excellent,
\(S\) is excellent, and \(q\) is a finite universal homeomorphism.
Both \(S^{\mathrm{sn}}\) and \(B\) are seminormal. Thus
Proposition~\ref{aca-seminormal-stages} gives exchange for
\(p_Sq\), and Lemma~\ref{aca-homeomorphisms} gives exchange for
\(q\). The composition law for the canonical maps reads
\[
 \tau_{p_Sq,\Lambda}
 =\tau_{q,\Lambda}\circ q^*\tau_{p_S,\Lambda}.
\]
It follows that \(q^*\tau_{p_S,\Lambda}\) is an equivalence.
Conservativity of \(q^*\), again by Lemma~\ref{aca-homeomorphisms},
proves the assertion.
\end{proof}

\subsection{Passage to general bases by continuity}

\begin{lemma}[Continuity of structural exchange]
\label{lem:continuity-structural-exchange}
For every noetherian finite-dimensional \(B\)-scheme \(S\), the map
\(\tau_{p_S,\Lambda}\) is an equivalence.
\end{lemma}

\begin{proof}
First suppose \(S=\operatorname{Spec}D\) is affine. Write \(D\) as
the filtered union of its finitely generated
\(\mathcal O(B)\)-subalgebras \(D_i\). Then
\(S=\lim_i S_i\), where \(S_i=\operatorname{Spec}D_i\) has finite
type over \(B\). The transition morphisms are affine and dominant
because the transition ring maps are injective. Write
\(p_i:S_i\to B\) for the structure maps.

We test \(\tau_{p_S,\Lambda}\) on smooth Tate generators in
\(\mathrm{DA}(S,\Lambda)\). Fix such a generator \(K\). Its smooth
separated scheme of finite presentation and its twist descend to a
common stage \(S_{i_0}\). Let \(K_i\) denote the resulting generators
for \(i\ge i_0\). Write \(F_i=F_{S_i,\Lambda}\),
\(F=F_{S,\Lambda}\), and denote the transferred units by
\(\mathbf1_i^{\mathrm{tr}}\) and \(\mathbf1_S^{\mathrm{tr}}\).
Extension of transfers commutes with pullback and takes smooth Tate
generators to the corresponding smooth Tate motives. In particular
\(F_i K_i\) are constructible and descend from \(F_{i_0}K_{i_0}\).

Continuity of \(\mathrm{DA}(-,\Lambda)\) and weak continuity of
\(\mathrm{DM}(-,\Lambda)\) apply to this inverse system. More
precisely, \cite[Proposition~4.3.4, Example~6.1.13, and
Theorem~11.1.24]{CD} computes maps from a constructible object
into any object pulled back from a common finite stage.
Applying this statement in all shifts gives its formulation for
mapping spectra. With coefficients \(\Lambda\) understood, it yields
\begin{align*}
 \operatorname{Map}_{\mathrm{DA}(S)}(K,p_S^*H\Lambda_B)
 &\simeq \operatorname*{colim}_{i\ge i_0}
   \operatorname{Map}_{\mathrm{DA}(S_i)}(K_i,p_i^*H\Lambda_B)\\
 &\simeq \operatorname*{colim}_{i\ge i_0}
   \operatorname{Map}_{\mathrm{DM}(S_i)}(F_iK_i,\mathbf1_i^{\mathrm{tr}})\\
 &\simeq \operatorname{Map}_{\mathrm{DM}(S)}
                 (FK,\mathbf1_S^{\mathrm{tr}})\\
 &\simeq \operatorname{Map}_{\mathrm{DA}(S)}(K,H\Lambda_S).
\end{align*}
In the first line the target descends because it is pulled back from
\(B\). In the second line we apply
Proposition~\ref{aca-finite-type} at each \(S_i\) and then use the
graph adjunction. This identification respects transition maps by the
composition law for exchange and the compatibility of \(F_i\) with
pullback. The third line uses weak continuity in \(\mathrm{DM}\)
with the transferred unit as target. The last line is the adjunction
defining \(H\Lambda_S\).

To identify the composite, represent a class at a finite stage by a
map \(K_i\to p_i^*H\Lambda_B\), allowing shifts of \(K_i\). Composing
with \(\tau_{p_i,\Lambda}\) gives the map adjoint to
\(F_iK_i\to\mathbf1_i^{\mathrm{tr}}\) in the second line. Pulling
this adjoint to \(S\) and taking its adjoint back gives precisely
the pullback of the original class followed by \(\tau_{p_S,\Lambda}\).
This follows from the definition of exchange as the mate of the
pullback comparison for extension of transfers. We have therefore
shown that \(\tau_{p_S,\Lambda}\) induces isomorphisms on all maps
from shifts of smooth Tate generators. These generators detect
equivalences, which proves the affine case.

For general \(S\), take a finite affine open cover. For an inclusion
\(j:U\hookrightarrow S\) of a member of the cover, the equality
\(\tau_{p_U,\Lambda}=\tau_{j,\Lambda}\circ j^*\tau_{p_S,\Lambda}\)
and the smooth exchange for \(j\) show that \(j^*\tau_{p_S,\Lambda}\)
is an equivalence. Restrictions to the cover are jointly conservative
by Nisnevich descent in \(\mathrm{DA}(S,\Lambda)\), proving the lemma.
\end{proof}

\begin{proof}[Proof of Theorem~\ref{thm:coefficient-exchange}]
Let \(f:T\to S\) be as in the theorem. Choose for \(P\) the positive
residue characteristics of \(S\). Both \(S\) and \(T\) are schemes
over the corresponding \(B\), and the hypothesis on \(R\) gives a
ring map \(\Lambda\to R\). Since \(p_T=p_Sf\), the composition
law gives
\[
 \tau_{p_T,\Lambda}=\tau_{f,\Lambda}\circ f^*\tau_{p_S,\Lambda}.
\]
The two structural exchanges are equivalences by
Lemma~\ref{lem:continuity-structural-exchange}. It follows that
\(\tau_{f,\Lambda}\) is an equivalence.

By Proposition~\ref{prop:cycle-tower-presentation}, derived extension
of scalars identifies \(H\Lambda_U\otimes_{\Lambda}^{\mathbf L}R\)
with \(HR_U\) for \(U=S,T\), compatibly with the canonical exchange
maps. Applying this functor to \(\tau_{f,\Lambda}\) gives exactly
\(\tau_{f,R}\), hence it is an equivalence. The exchange maps are
unital maps of commutative algebras, and the forgetful functor detects
equivalences. This proves base change stability of \(HR\), compatibly
with multiplication and derived extension of coefficients.
\end{proof}

\section{Proof of the module comparison theorem}
\label{sec:proof-module-comparison}

We now apply Theorem~\ref{thm:coefficient-exchange} to the comparison
\(\mu_{S,R}:A_{S,R}\to HR_S\) of Proposition~\ref{prop:counit-comparison}.
Fix the set of primes \(P\), the arithmetic base \(B\), and the
coefficient ring \(\Lambda\) of \eqref{aca-arithmetic-base}.

\subsection{The comparison of coefficient algebras}

\begin{lemma}[The comparison over the residue fields of the arithmetic base]
\label{lem:residue-field-comparison}
For every residue field \(k\) of \(B\), the canonical map
\(\mu_{k,\mathbb Z}:A_{k,\mathbb Z}\to H\mathbb Z_k\) is an equivalence.
The same conclusion holds with coefficients in \(\Lambda\).
\end{lemma}

\begin{proof}
The residue fields of \(B\) are \(\mathbb Q\) and the fields
\(\mathbb F_p\) for \(p\in P\). Over each of these perfect fields the
unit \(\mathbf1_k\to H\mathbb Z_k\) induces an equivalence
\(s_0(\mathbf1_k)\simeq H\mathbb Z_k\)
\cite[Example~8.2.2(3) and Section~10.4]{Levine},
as recalled in \cite[proof of Theorem~6.3]{BEM}.
We must identify the underlying map of \(\mu_{k,\mathbb Z}\) with
this equivalence. The fiber of \(\mathbf1_k\to s_0(\mathbf1_k)\)
is \(1\)-effective, and \(H\mathbb Z_k\), being a zero slice, is
right orthogonal to \(1\)-effective spectra. Precomposition therefore
gives an equivalence of mapping spaces
\[
 \operatorname{Map}_{\mathrm{SH}(k)}(s_0(\mathbf1_k),H\mathbb Z_k)
 \simeq
 \operatorname{Map}_{\mathrm{SH}(k)}(\mathbf1_k,H\mathbb Z_k).
\]
Under \(\underline A_{k,\mathbb Z}\simeq s_0(\mathbf1_k)\), the map
\(\mu_{k,\mathbb Z}\) restricts to the unit by
Proposition~\ref{prop:counit-comparison}. Its underlying map is
therefore the field equivalence. Forgetting algebra and ordinary linear
structures detects equivalences. Derived extension of scalars gives
the assertion over \(\Lambda\) by
Propositions~\ref{prop:cycle-tower-presentation} and~\ref{prop:counit-comparison}.
\end{proof}

\begin{lemma}[The comparison over the arithmetic base and its pullbacks]
\label{lem:arithmetic-algebra-comparison}
The map \(\mu_{B,\Lambda}\) is an equivalence. For every noetherian
finite-dimensional \(B\)-scheme \(u:S\to B\), the canonical map
\(\mu_{S,\Lambda}:A_{S,\Lambda}\to H\Lambda_S\) is an equivalence
of commutative algebras in \(\mathrm{DA}(S,\Lambda)\).
\end{lemma}

\begin{proof}
For a residue point \(b:\operatorname{Spec}k\to B\), identify
\(b^*A_{B,\Lambda}\) with \(A_{k,\Lambda}\) using
\eqref{eq:model-A-properties}. Naturality of the comparison,
\eqref{eq:comparison-naturality}, then reads
\[
 \tau_{b,\Lambda}\circ b^*\mu_{B,\Lambda}\simeq\mu_{k,\Lambda}.
\]
Lemma~\ref{lem:residue-field-comparison} identifies the right side
as an equivalence. Theorem~\ref{thm:coefficient-exchange} gives the
equivalence \(\tau_{b,\Lambda}\) for every such \(b\), including
the generic point of \(B\). Thus \(b^*\mu_{B,\Lambda}\) is an
equivalence for every point. These pullbacks are jointly conservative
in \(\mathrm{DA}(B,\Lambda)\)
\cite[Proposition~4.3.17 and Example~6.1.13]{CD}, which proves the
first assertion.

For \(u:S\to B\), the same naturality identifies \(\mu_{S,\Lambda}\)
with the composite
\[
 A_{S,\Lambda}\simeq u^*A_{B,\Lambda}
 \xrightarrow{\,u^*\mu_{B,\Lambda}\,}u^*H\Lambda_B
 \xrightarrow{\,\tau_{u,\Lambda}\,}H\Lambda_S.
\]
The first map after the identification is the pullback of the
equivalence just proved. The last map is an equivalence by
Theorem~\ref{thm:coefficient-exchange}.
\end{proof}

\subsection{Reconstruction and extension of coefficients}

\begin{lemma}[Reconstruction over \(\Lambda\)]
\label{lem:module-reconstruction}
For every noetherian finite-dimensional \(B\)-scheme \(S\), the functor
\[
 \Theta_{S,\Lambda}:
 \operatorname{Mod}_{A_{S,\Lambda}}(\mathrm{DA}(S,\Lambda))
 \longrightarrow \mathrm{DM}(S,\Lambda)
\]
is a symmetric monoidal equivalence.
\end{lemma}

\begin{proof}
For each finite-type morphism \(f:Y\to X\) between these \(B\)-schemes,
Theorem~\ref{thm:coefficient-exchange} gives
\(f^*H\Lambda_X\simeq H\Lambda_Y\) by the canonical exchange map.
This is the hypothesis of the reconstruction theorem of
Cisinski--D\'eglise \cite[Proposition~11.4.7]{CD}. Their comparison
from modules over \(H\Lambda_S\) in \(\mathrm{SH}(S)\) to
\(\mathrm{DM}(S,\Lambda)\) is therefore an equivalence on homotopy
categories. The comparison is an exact functor of stable categories,
so full faithfulness in all shifts gives an equivalence on mapping
spectra. Essential surjectivity follows from the homotopy-category
statement.

To identify the functor, apply
Lemma~\ref{lem:zero-slice-module-transitivity} and transport the module
structure along \(\mu_{S,\Lambda}\), which is an equivalence by
Lemma~\ref{lem:arithmetic-algebra-comparison}. The reconstruction functor
uses the counit \(F_{S,\Lambda}(H\Lambda_S)\to\mathbf1_\Lambda\).
Its composite with \(F_{S,\Lambda}(\mu_{S,\Lambda})\) is the adjoint
of \(\mu_{S,\Lambda}\). Extension of scalars along this composite is
exactly the symmetric monoidal functor \(\Theta_{S,\Lambda}\) of
\eqref{eq:introduction-comparison}.
\end{proof}

\begin{lemma}[Derived extension of coefficients]
\label{lem:module-scalar-extension}
For a commutative \(\Lambda\)-algebra \(R\), derived extension from
\(\Lambda\) to \(R\) takes the pair
\((\mu_{S,\Lambda},\Theta_{S,\Lambda})\) to
\((\mu_{S,R},\Theta_{S,R})\).
\end{lemma}

\begin{proof}
Proposition~\ref{prop:cycle-tower-presentation} identifies
\(H\Lambda_S\otimes_\Lambda^{\mathbf L}R\) with \(HR_S\) as a
commutative algebra. By construction
\(A_{S,\Lambda}\otimes_\Lambda^{\mathbf L}R\simeq A_{S,R}\), and
Proposition~\ref{prop:counit-comparison} identifies the resulting
algebra map with \(\mu_{S,R}\).

For the categories, extension of scalars tensors a presentable
\(D(\Lambda)\)-linear category with \(D(R)\) over \(D(\Lambda)\).
Applying it to the linear presheaf construction and its Nisnevich,
\(\mathbb A^1\), and Tate localizations gives
\(\mathrm{DA}(S,R)\) and \(\mathrm{DM}(S,R)\) from the respective
categories with \(\Lambda\)-coefficients. The smooth generators and
the maps imposing these localizations are carried to those with
\(R\)-coefficients. On modules the same extension takes the category
of \(A_{S,\Lambda}\)-modules to that of \(A_{S,R}\)-modules.
The graph functor carries over to \(F_{S,R}\), and the relative
tensor product defining \(\Theta_{S,\Lambda}\) carries over to the
relative tensor product defining \(\Theta_{S,R}\). These descriptions
use derived tensor products for arbitrary \(R\).
\end{proof}

\begin{proof}[Proof of Theorem~\ref{thm:main}]
Take \(P\) to be the positive residue characteristics of \(S\).
Then \(S\to\operatorname{Spec}\mathbb Z\) factors through \(B\), and
\(\Lambda\to R\) exists by the coefficient hypothesis, as in
\eqref{aca-arithmetic-base}. Lemmas~\ref{lem:arithmetic-algebra-comparison}
and~\ref{lem:module-reconstruction} give the algebra and module
equivalences over \(\Lambda\). Lemma~\ref{lem:module-scalar-extension}
gives both equivalences over \(R\).
For a smooth \(S\)-scheme \(X\), \eqref{eq:comparison-free-modules}
identifies the image of the free module on the ordinary linear motive
of \(X\) with its motive with transfers \(M_R(X)\).
Compatibility with pullback follows from
\eqref{eq:comparison-naturality} and the relative tensor product
defining \(\Theta_{S,R}\). The same formula gives compatibility with
further derived extension of coefficients.
\end{proof}

\end{document}